\documentclass[12pt]{article}
\usepackage{amsmath, amssymb, amsthm}

\newcommand{\comment}[1]{}
\newcommand{\n}{\par\noindent}

\newcommand{\sn}{\par\smallskip\noindent}
\newcommand{\mn}{\par\medskip\noindent}
\newcommand{\bn}{\par\bigskip\noindent}
\newcommand{\pars}{\par\smallskip}

\newcommand{\bR}{\mathbf R}

\theoremstyle{definition}
\newtheorem{definition}{Definition}

\usepackage{xspace}

\makeatletter

\newcommand{\intro}[1]{\textbf{#1}}

\newcommand{\Pa}[1]{\bigl(#1\bigr)}

\providecommand{\set}[1]{\ensuremath{\bigl\{\,#1\,\bigr\}}}

\newcommand{\field}[1]{\ensuremath{\mathbb{#1}}}
\newcommand{\Nat}{\field{N}}
\newcommand{\Nats}{\Nat^\star}

\newcommand{\Zed}{\field{Z}}

\newcommand{\et}{\ \&\ }

\newcommand{\rfield}[1]{\ensuremath{\mathbf{#1}}}
\newcommand{\KF}{\ensuremath{\field{K}}}
\newcommand{\Kf}{\ensuremath{\field{K}}}
\newcommand{\kf}{\ensuremath{\rfield{k}}}
\newcommand{\KS}{\ensuremath{\KF^\star}}
\newcommand{\ks}{\ensuremath{\kf^\star}}
\newcommand{\FF}{\ensuremath{\field{F}}}
\newcommand{\HF}{\ensuremath{\field{H}}}
\newcommand{\LF}{\ensuremath{\field{L}}}
\newcommand{\ff}{\ensuremath{\rfield{f}}}
\newcommand{\Oh}{\Dedekind O}
\newcommand{\extension}[2]{#1/#2}

\newcommand{\KH}{\ensuremath{\KF^{\mathrm H}}}
\newcommand{\videal}{\ensuremath{\mathcal{M}}}
\newcommand{\vg}{G}
\newcommand{\val}{\ensuremath{v}}
\newcommand{\ball}[1]{\ensuremath{\boldsymbol{\Lambda}_{#1}}}
\newcommand{\ballc}{\ensuremath{\ball{c}}}
\newcommand{\La}{\ensuremath{\ball{a}}}
\newcommand{\res}[1]{\overline{#1}}
\newcommand{\factor}{f}%

\newcommand{\vcomptwo}[2][]{%
{#2}%
  \def\tempa{}%
     \def\tempb{#1}%
      \ifx\tempa\tempb%
      \else%
       [#1]%
     \fi%
  }

\newcommand{\vcomp}[1][]{\ensuremath{A%
    \def\tempa{}%
    \def\tempb{#1}%
    \ifx\tempa\tempb%
    \else%
      [#1]%
    \fi%
  }}
\newcommand{\vcompf}{\ensuremath{\mathcal{A}}}
\newcommand{\vcompL}{\ensuremath{\vcomp[\Lambda]}}
\newcommand{\vcompG}{\ensuremath{\vcomp[\Gamma]}}
\newcommand{\vcompb}[1][]{\ensuremath{B
    \def\tempa{}%
    \def\tempb{#1}%
    \ifx\tempa\tempb%
    \else%
      [#1]%
    \fi%
  }}
\newcommand{\vcompbp}[1]{\ensuremath{B'[#1]}}
\newcommand{\vcompbf}{\ensuremath{\mathcal{B}}}
\newcommand{\vcompbL}{\ensuremath{\vcompb[\Lambda]}}
\newcommand{\vcompbG}{\ensuremath{\vcompb[\Gamma]}}
\newcommand{\vcompcf}{\mathcal C}

\newcommand{\vcompcG}{\vcomptwo{C}[\Gamma]}
\newcommand{\vring}[1][]{\ensuremath{\mathcal{O}%
    \def\tempa{}%
    \def\tempb{#1}%
    \ifx\tempa\tempb%
    \else%
      [#1]%
    \fi%
  }}
\newcommand{\vringL}{\ensuremath{\vring[\Lambda]}}
\newcommand{\vringG}{\ensuremath{\vring[\Gamma]}}

\newcommand{\Dedekind}[1]{\breve{#1}}

\newcommand{\vgh}{\Dedekind\vg}
\DeclareMathOperator*{\smallop}{o}
\newcommand{\smallo}[1]{\smallop(#1)}
\newcommand{\cut}[2]{\bigl(#1,\, #2\bigr)}

\newcommand{\gp}{\gamma^+}
\newcommand{\gm}{\gamma^-}

\newcommand{\lp}{\lambda^+}
\newcommand{\lm}{\lambda^-}

\newcommand{\anu}{a_\nu}
\newcommand{\amu}{a_\mu}
\newcommand{\isequence}{\ensuremath{(a_\nu)_{\nu\in I}}}

\newcommand{\dl}
{\mathbin{-\mkern-11mu^{\mathrm{L}}\mkern2mu}}

\def\ar{\mathbin{+\mkern-5mu^{\mathrm{R}}}}

\newcommand{\ig}[1]{\ensuremath{\widehat{#1}}}

\newcommand{\fsum}{\mathcal{S}}
\newcommand{\xs}{\fsum x}
\DeclareMathOperator{\lt}{lt}
\DeclareMathOperator*{\supp}{supp}
\newcommand{\KG}{\mbox{}^G\kf}
\newcommand{\kG}{\kf((\vg))}
\newcommand{\kGf}{\kf((\vg,\factor))}
\newcommand{\kg}{\kf(\vg)}
\newcommand{\kgring}{\kf[\vg]}

\newcommand{\wloG}{w.l.o.g\mbox{.}\xspace}
\newcommand{\Wlog}{W.l.o.g\mbox{.}\xspace}
\newcommand{\ie}{i.e\mbox{.}\xspace}

\newcommand{\rhs}{R.H.S\mbox{.}\xspace}
\newcommand{\lub}{l.u.b\mbox{.}\xspace}
\newcommand{\glb}{g.l.b\mbox{.}\xspace}

\newcommand{\wtoc}{W.T.o.C\mbox{.}\xspace}
\newcommand{\towcomp}{T.o.C\mbox{.}\xspace}
\newcommand{\CWToC}{C.W.T.o.C\mbox{.}\xspace}
\newcommand{\CToC}{C.T.o.C\mbox{.}\xspace}
\newcommand{\WToC}{W.T.o.C\mbox{.}\xspace}
\newcommand{\ToC}{T.o.C\mbox{.}\xspace}

\swapnumbers
\theoremstyle{plain}
\newtheorem{lemma}{Lemma}
\newtheorem{thm}[lemma]{Theorem}
\newtheorem{proposition}[lemma]{Proposition}
\newtheorem{corollary}[lemma]{Corollary}
\newtheorem{fhypothesis}[lemma]{Fundamental hypothesis}
\theoremstyle{remark}
\newtheorem{remark}[lemma]{Remark}
\newtheorem{example}[lemma]{Example}

\newtheorem*{example*}{Example}
\newtheorem*{examples*}{Examples}

\newcounter{examplecounter}[section]
\newtheorem{mainexample}[lemma]{Main example}

\newtheorem{o@claim}[lemma]{Claim}
\swapnumbers
\newtheorem{i@claim}{Claim}
\swapnumbers
  {\setcounter{answer}{\value{lemma}}
    \begin{@question}}%
  {\end{@question}}
\theoremstyle{definition}

\newtheorem{definizione}[lemma]{Definition}
\newtheorem{@question}[lemma]{Question}
\newtheorem{notation}[lemma]{Notation}
\newenvironment{hypothesis*}[1][Hypothesis]
{\par\noindent\textbf{#1.}}%
{\par}
\newenvironment{axiom*}[1][Axiom]%
{\par\noindent\textbf{#1.}}%
{\par}
\newcounter{proof}
\newif\if@inproof
\@inprooffalse
\let\oldproof\proof
\let\oldendproof\endproof
\renewcommand{\proof}[1]{%
  \refstepcounter{proof}%
  \@inprooftrue%
  \oldproof#1}
\def\endproof{\oldendproof\@inprooffalse}
\newenvironment{claim}[1][]{%
  \if@inproof\begin{i@claim}[#1]%
  \else\begin{o@claim}[#1]%
  \fi}{%
  \if@inproof\end{i@claim}%
  \else\end{o@claim}%
  \fi}%
\numberwithin{i@claim}{proof}
\renewcommand{\thei@claim}{\arabic{i@claim}}

\numberwithin{equation}{section}
\numberwithin{lemma}{section}
\numberwithin{answer}{section}

\swapnumbers

\makeatother

\newcommand{\adresse}{\par\bigskip\small \rm
A.\ Fornasiero:\par
Institut f\"ur mathematische Logik
und Grundlagen der Mathematik \par
Albert-Ludwigs-Universit\"at,
D-79104 Freiburg, Germany.\par
Email: antongiulio.fornasiero@googlemail.com
\pars
F.-V. Kuhlmann:\par
Research Center for Algebra, Logic and Computation,\par
University of Saskatchewan, S7N 5E6, Canada.\par
Email: fvk@math.usask.ca
\pars
S. Kuhlmann:\par
Research Center for Algebra, Logic and Computation,\par
University of Saskatchewan, S7N 5E6, Canada.\par
Email: skuhlman@math.usask.ca}
\title{
Towers of complements to valuation rings and truncation closed
embeddings of valued fields
\footnote{2000 Mathematics Subject Classification: Primary: 12J10,
12J15, 12L12, 13A18; \n Secondary: 03C60, 12F05, 12F10, 12F20. \n {\it
Key words:} completion of an ordered group, valued field,
fields of generalized power series, truncation closed embedding,
complement to valuation ring.\n
{\bf Research partially supported
by NSERC Discovery Grants.}
}}

\author{Antongiulio Fornasiero, Franz-Viktor Kuhlmann,\\ Salma Kuhlmann
}
\date{April 15, 2008}
\begin{document}
\maketitle
\begin{abstract}
We study necessary and sufficient conditions for a valued field~$\KF$
with value group $G$ and residue field $\kf$ (with char $\KF$ = char
$\kf$) to admit a truncation closed embedding in the field of
generalized power series $\kf((G, f))$ (with factor set~$f$). We show
that this is equivalent to the existence of a family ({\it tower of
complements}) of $\kf$-subspaces of $\KF$ which are complements of the
(possibly fractional) ideals of the valuation ring. If $\KF$ is a
Henselian field of characteristic~0 or, more generally, an algebraically
maximal Kaplansky field, we give an intrinsic construction of
such a family which does not rely on a given truncation closed embedding.
We also show that towers of complements and truncation closed embeddings
can be extended from an arbitrary field to at least one of its maximal
immediate extensions.
\end{abstract}
\section{Introduction}
Truncation closed embeddings of valued fields in fields of
generalized power series (see Section \ref{valf} for
definitions and notations) were introduced by Mourgues and Ressayre
in their investigation of integer parts of ordered fields. An {\bf
integer part} (IP for short) $Z$ of an ordered field $\KF$ is a
discretely ordered subring, with $1$ as the least positive element,
and such that for every $x\in \KF$, there is a $z\in Z$ such that
$z\leq x < z+1$. The interest in studying these rings originates from
Shepherdson's work in
\cite{S}, who showed that IP's of real closed fields are precisely
the models of a fragment of Peano Arithmetic called Open Induction.
In \cite{MOU-RES}, the authors establish the existence of an
IP for any real closed field $\KF$ as follows (see Section \ref{valf}
for definitions and notations): Let $v$ be the natural valuation on
$\KF$. Denote by $\kf$ the residue field and by $G$ the value group of
$\KF$. Fix a residue field section $\iota$ (we will assume that $\iota$
is the identity), and a value group section $t:\vg\to\KS$. [M--R] show
that there is an order preserving embedding $\varphi$ of $\KF$ in the
field of generalized power series $\kf ((G))$ such that $\varphi (\KF)$
is a truncation closed subfield. They observe that for the field $\kf
((G))$, an integer part is given by $\kf ((G^{<0})) \oplus \Zed$, where
$\kf ((G^{<0}))$ is the (non-unital) $\kf$-algebra of power series with
negative support. It follows that for any truncation closed subfield $F$
of $\kf ((G))$, an integer part is given by $Z_F=(\kf ((G^{<0}))\cap
F)\oplus \Zed$. Finally $\varphi^{-1}(Z_F)$
is an integer part of $\KF$ if we take $F=\varphi (\KF)$.
Let $A$ be a $\kf$-subspace of $\KF$ which is also an
additive complement to the valuation ring~$\vring$ of $\KF$: we will
call $A$ a {\bf $\kf$-complement of~$\vring$}. We say that $A$ is {\bf
multiplicative} if $A \cdot A \subseteq A$ (i.e. if $A$ is a
$\kf$-algebra). A multiplicative $\kf$-complement of~$\vring$ will be
called a {\bf $\kf$-algebra complement of~$\vring$}. Since
$\kf((G^{<0}))\cap F$ is clearly a multiplicative $\kf$-complement of
the valuation ring of $F$, $\varphi^{-1}(\kf((G^{<0}))\cap F)$ is a
multiplicative $\kf$-complement of~$\vring$. We call an integer part $Z$
of $\KF$ (respectively a multiplicative $\kf$-complement of~$\vring$)
obtained in this way from a truncation closed embedding
a {\bf truncation integer part} of $\KF$ (respectively a {\bf truncation
$\kf$-algebra complement of~$\vring$}, or {\bf truncation $\kf$-algebra}
for short). In this terminology, it follows in particular from
\cite{MOU-RES} that every real closed field $\KF$ with residue field
$\kf$ and valuation ring $\vring$ admits a truncation IP and a
truncation $\kf$-algebra complement of~$\vring$.

\mn
In light of these results, we asked in Remarks 3.2 and 3.3
of \cite{B-K-K} more generally for necessary and sufficient
conditions on the valued field $\KF$
for the existence of
$\kf$-algebra complements of $\vring$. We also asked whether
such complements are always truncation $\kf$-algebra complements.
A valued
field $\KF$ with residue field $\kf$ and value group $G$ is called a
{\bf Kaplansky field} if $\kf$ and $G$
satisfy a pair of conditions called
Kaplansky's ``Hypothesis A'' (\cite{Kap}; page 312 statements (1) and
(2)).
The conditions
are void if the characteristic $\kf$ is $0$, whereas
if the characteristic of $\kf$ is
$p>0$, then Hypothesis A holds if and only if $G$ is
$p$-divisible and $\kf$ does not admit any extensions of
degree divisible by $p$.
Kaplansky's original result \cite{Kap} Theorem 6\footnote{There is a
misprint on line 2 of the statement of the Theorem; $K$ should be
replaced by ${\mathfrak K}$.}
yields that a Kaplansky field $\KF$
(with char $\KF$ = char
$\kf$) can be embedded in the power series field $\kf((G, f))$, possibly
with a suitable choice of a factor set $f$
(cf.\ Definition
\ref{psfwfactorset}). In view of this, it is natural to ask
 more generally for necessary and sufficient
conditions on a valued field $\KF$
for the
existence of truncation
closed embeddings of $\KF$ in $\kf ((G, f))$.
This paper addresses these questions.
\mn
{\it Unless otherwise stated, we assume throughout that the valued field
$\KF$ admits a value group section and a residue field section. In
particular, we only deal with the ``equal characterisitic'' case, i.e.,
char $\KF$ = char $\kf$.}
\mn
The main tool in our investigations is the
following concept. A tower of complements of the valuation ring $\vring$
of $\KF$ is a family $\vcompf = \set{\vcomp[\Lambda]:\Lambda\in
\vgh}$ of $\kf$-subspaces of $\Kf$ (indexed by the order completion
$\vgh$ of the value group $G$; cf.\ Definition \ref{completion})
satisfying certain natural properties (cf.\ Definition \ref{wtoc}). We
show (cf.\ Theorem \ref{THM:EMBEDDING}) that the existence of such a
family is a necessary and sufficient condition for a valued field $(\KF,
v)$ to admit a truncation closed embedding in the field of generalized
power series $\kf((G, f))$, and that there is a one-to-one correspondence
between towers of complements for $\KF$ and truncation closed embeddings
of $\KF$ in $\kf((G, f))$ (cf.\ Corollary \ref{COR:EMBEDDING}.)
A tower of complements $\vcompf$ is uniquely determined by $\vcomp[0^-]
\in \vcompf$ which is a $\kf$-algebra complement of~$\vring$ (cf.\
Corollary \ref{corto7}). In particular, given $A$ a $\kf$-algebra
complement of~$\vring$, we conclude that $A$ is a truncation
$\kf$-algebra complement if and only if there is a tower of complements
$\vcompf$ such that $\vcomp = \vcomp[0^-]$. In Section
\ref{oldopenquestion} we analyze when such an $\vcompf$ can be
constructed for a given $\vcomp$.

\mn
In Section \ref{building}, we analyze the procedure for extending towers
of complements (T.o.C.). We start with the subfield of rational series
$\kf(G)\subseteq \KF$ and proceed by induction, building T.o.C.'s for
larger and larger subfields of $\KF$. The field of rational
series $\kf(G)$ has a \ToC (Section~\ref{basicase}). We proceed by
induction: we assume that we already built a \ToC ${\cal A}$ for a
subfield $K$ of $\KF$, such that $\kf(G)\subseteq K$. Note that $\KF$ is
an immediate extension of $K$. Let $a \in \KF\setminus K$. We want to
extend the T.o.C.\ ${\cal A}\>$ to a T.o.C.\ ${\cal B}\>$ for $K(a)$.
Let $(a_\nu)_{\nu\in I}$ be a pseudo Cauchy sequence in $K$ with limit
$a$. We assume that $a$ satisfies one of the two conditions a) or b) of
the Fundamental Hypothesis~\ref{FH}. In case b), the algebraic case, we
can extend ${\cal A}$ to a T.o.C.\ ${\cal B}$ for $K(a) = K[a]$ if
certain conditions are satisfied; the extension ${\cal B}$ is
defined in Definition~\ref{DEF:EXT-RING}. In particular, we can extend
${\cal A}$ to the Henselization of $K$ (Lemma~\ref{LEM:TOWER-HENSEL}).
In case a), the transcendental case, we can extend ${\cal A}$ to a
T.o.C.\ ${\cal B}$ for $K[a]$ (Corollary~\ref{COR:FIELD-MULT}); note
that in this case $K[a]$ is a ring, not a field. Then we extend ${\cal
B}$ further to a T.o.C.\ for $K(a)$, the quotient field of $K[a]$.

\mn
Putting these steps together, we prove that every Henselian field of
residue characteristic 0 has a T.o.C. . In fact, start with \kf(G), pass
to the Henselization, add a transcendental element $a$ satisfying a),
pass again to the Henselization, and so on. Here we use that Henselian
fields of residue characteristic $0$ do not admit proper
immediate algebraic extensions. But in the case of positive
residue characteristic, one has to deal with such extensions. In
Section~\ref{PRC} we prove that towers of complements on a
given field of positive characteristic can be extended to {\it at least
one} maximal immediate algebraic extension and thus to {\it at least
one} maximal immediate extension (Theorem~\ref{extamKf}). It follows
from our results of Sections~\ref{5.3} and \ref{PRC} that, in the equal
characterisitc case, algebraically maximal Kaplansky fields admit towers
of complements and thus also truncation closed embeddings in power
series fields (cf.\ Theorem \ref{buildhens}, Theorem~\ref{buildkapl} and
Corollary~\ref{tcepc}). Without the condition ``algebraically maximal'',
the existence of truncation closed embeddings can in general not be
expected (cf.\ examples in \cite{F} and \cite{Ku2}).

\mn
Note that Fornasiero (\cite{F}; Theorem 5.1) already showed the
existence of truncation closed embeddings
for Henselian fields of residue characteristic 0, and has
indicated the same result (in the equal positive characteristic case)
for algebraically maximal Kaplansky fields (\cite{F}; paragraph
following Theorem 8.12), by generalizing the approach of
\cite{MOU-RES}. But in this paper we prove it through an intrinsic
construction of towers of complements. This approach allows us to obtain
an even stronger result. Namely, Theorem~\ref{extamKf} implies that a
truncation closed embedding of a field of positive characteristic can be
extended to a truncation closed embedding of at least one of its maximal
immediate algebraic extensions and at least one of its maximal immediate
extensions, even if the field is not a Kaplansky field. In that case,
the truncation closed embedding may not be extendable to all such
extensions, as we show in an example. This means that for such fields,
there are maximal immediate extensions that are ``better'' than others.
It should definitely be interesting to study their properties, both from
an algebraic and from a model theoretic point of view.

\mn
We conclude with the following remark and open question:
There exist valued fields (with sections for the value group $G$ and the
residue field $\kf$ ) that admit a valuation preserving embedding
(compatible with the sections) in $\kf((G, f))$, but admit {\it no}
truncation closed embedding (cf.\ examples in \cite{F} and \cite{Ku2}).
We do not know whether such fields can be algebraically maximal.

\par\bigskip
\section{Preliminaries}\label{prelim}
\subsection{Dedekind cuts on ordered groups}\label{ordg}
Let $O$ be an ordered set.
A \intro{cut} $\cut{\Lambda^L}{\Lambda^R}$ of $O$ is a partition of $O$
into two subsets $\Lambda^L$ and $\Lambda^R$, such that, for every
$\lambda \in\Lambda^L$ and $\lambda' \in \Lambda^R$, $\lambda < \lambda'$.
We will denote with $\Dedekind{O}$ the set of cuts $\Lambda:=\cut{\Lambda^L}
{\Lambda^R}$ of the order $O$ (including $-\infty := \cut{\emptyset}
{O}$ and $+\infty := \cut{O}{\emptyset}$).
\sn
Unless specified otherwise, small Greek letters
$\gamma, \lambda, \cdots$ will range among elements of~$O$,
capital Greek letters $\Gamma,\Lambda, \cdots$ will range among elements
of~$\Dedekind{O}$.
\sn
Given $\gamma\in O$,
\[\begin{aligned}
\gamma^- &:= \cut{(-\infty, \gamma)}{[\gamma, +\infty)} &\text{and}\\
\gamma^+ &:= \cut{(-\infty, \gamma]}{(\gamma, +\infty)}
\end{aligned}\]
are the cuts determined by it.
Note that $\Lambda^R$ has a minimum $\lambda$ if and only if $\Lambda
= \lambda^-$.
Dually, $\Lambda^L$ has a maximum $\gamma$ if and only if $\Lambda
= \gamma^+$.
\mn
The \intro{ordering} on $\Oh$ is given by $\Lambda\leq \Gamma$ if and only if
$\Lambda^L\subseteq\Gamma^L$ (or, equivalently,
$\Lambda^R\supseteq\Gamma^R$).
To simplify the notation, we will sometimes write $\gamma < \Lambda$
as a synonym of $\gamma\in\Lambda^L$, or equivalently $\gm < \Lambda$,
or equivalently $\gp \leq \Lambda$.
Similarly, $\gamma > \Lambda$ if and only if $\gamma\in\Lambda^R$, or equivalently $\gp > \Lambda$.
Hence, we have $\gm < \gamma < \gp$.
\mn
An ordered set $O$ is \intro{complete} if for every $S \subseteq O$, the
\lub and the \glb of $S$ exist.
Note that if $O$ is any ordered set, then $\Oh$ is complete.
Given a subset $S\subseteq O$, $S^+\in\Oh$ is the smallest cut
$\Lambda$ such that $S \subseteq \Lambda^L$,
and $S^-$ is the largest cut $\Gamma$ such that $S \subseteq \Lambda^R$.
Note that $S^+ = -\infty$ if and only if $S$ is empty, and $S^+ =
+\infty$ if and only if $S$ is unbounded.
Note also that $ S^+ = \sup\set{\gp:\gamma\in S}$,%
\footnote{The supremum on the \rhs is taken in the ordered set $\Oh$.}
and $S^+ > \gamma$ for every $\gamma\in S$.
Moreover, $\Lambda = (\Lambda^L)^+$.
\bn
Now let $\vg$ be an ordered Abelian group.
Given $\Lambda, \Gamma \in \vgh$, their {\it left} \intro{sum} is the
cut
\[
\Lambda + \Gamma := \set{\lambda + \gamma : \lambda < \Lambda,
\gamma < \Gamma}^+.
\]
We also define {\it right} \intro{sum}, as
\[
\Lambda \ar \Gamma
:= \set{\lambda + \gamma: \lambda > \Lambda, \gamma >\Gamma}^-.
\]
Given $\gamma \in \vg$, we write
\[
\gamma + \Lambda := \cut{\set{\gamma + \lambda:
\lambda\in\Lambda^L}}{\set{\gamma + \lambda': \lambda' \in \Lambda^R}}.
\]
One can verify that $\gamma + \Lambda = \gp + \Lambda = \gm \ar
\Lambda$, and that:
\[\begin{array}{r@{\;\;\extracolsep{1ex}+}l@{\;\;=}r@{\;\;\ar}l@{\;\;=}l}
\Lambda & (+\infty) & \Lambda & (+\infty) & +\infty,\\
\Lambda & (-\infty) & \Lambda & (-\infty) & -\infty,\\
+\infty & (+\infty) & +\infty & (+\infty) & + \infty,\\
-\infty & (-\infty) & -\infty & (-\infty) & - \infty.
\end{array}\]
Note also that $\Lambda + \Gamma \leq \Lambda \ar \Gamma$.
\begin{remark} \label{completion}
$(\vgh, \leq)$ is a complete linear order, and moreover
$(\vgh, +, \leq)$ is an ordered commutative monoid, with neutral
element~$0^+$, that is, if $\alpha \leq \beta$,
then $\alpha + \gamma \leq \beta + \gamma$.
Similarly $(\vgh, \ar, \leq)$ is an ordered commutative monoid,
with neutral element~$0^-$.
The map $\phi^+$ (resp.\ $\phi^-$) from
$(\vg, \leq, 0, +)$ to $(\vgh, \leq, 0^+, +)$
(resp.\ to $(\vgh, \leq, 0^-, \ar)$) sending $\gamma$ to $\gp$
(resp.\ to $\gm$) is a homomorphism of ordered monoids.
The anti-isomorphism $-$ of $(\vg, \leq)$, sending $\gamma$ to
$-\gamma$,
induces an anti-isomorphism (with the same name $-$) between
$(\vgh, \leq, +)$ and $(\vgh, \geq, \ar)$,
sending $\Lambda$ to  $\cut{-\Lambda^R}{-\Lambda^L}$.
Hence, all theorems about $+$ have a dual statement about $\ar$.
\end{remark}
\begin{remark}
Note that
$-(\gp) = (-\gamma)^-$, and
$-(\gm) = (-\gamma)^+$.
\end{remark}
\begin{definizione}
Given $\Lambda, \Gamma\in\vgh$, define their (right)
difference $\Lambda - \Gamma$ in the following way:
\[
\Lambda - \Gamma := \set{\lambda - \gamma : \lambda >
\Lambda, \gamma < \Gamma}^-.
\]
\end{definizione}
The following Lemma is easily proved (see Lemma 2.11 and Remark
3.6 of
\cite{FOR-MAM}).
\begin{lemma}\label{LEM:SUM-DIFF}
\begin{enumerate}
\item $\Lambda < \Theta$ if and only if $\Lambda - \Theta < 0$.
\item $\Lambda \geq \Theta$ if and only if $\Lambda - \Theta > 0$.
\item $\Lambda > - \Theta$ if and only if $\Lambda + \Theta > 0$.
\item $\Lambda < \Gamma + \Theta$ if and only if $\Lambda - \Gamma < \Theta$.
\item $\Lambda \geq \Gamma + \Theta$ if and only if $\Lambda -
\Gamma \geq \Theta$.
\item $\Lambda \geq (-\Gamma) + \Theta$ if and only if $\Lambda \ar
\Gamma \geq \Theta$.
\end{enumerate}
\end{lemma}
\begin{definizione}
Given $n\in\Nat$, define
\[\Lambda - n\Gamma := \Lambda \underbrace{-\Gamma -
\dotsb -\Gamma}_{n \text{ times}},\mbox{ and }
\Lambda + n\Gamma := \Lambda \underbrace{+ \Gamma +
\dotsb \Gamma}_{n\text{ times}}.\]
In particular, $\Lambda - 0 \Gamma = \Lambda + 0 \Gamma = \Lambda$.
Moreover, given $n\in\Nats$, define
\[
n\Gamma := \underbrace{\Gamma + \dotsb \Gamma}_{n\text{ times}},
\mbox{ and }
(-n)\Gamma := -(n\Gamma) = \underbrace{-\Gamma -
\dotsb -\Gamma}_{n \text{ times}}.\]
\end{definizione}
The following technical results will be used throughout the paper, and
are given without proof
(see Proposition 3.15 and Corollary 3.16 of \cite{FOR-MAM}).
\begin{proposition}\label{PROP:N-GAMMA}
$(\Lambda - n\Gamma) + n\Gamma \leq \Lambda \leq (\Lambda + n\Gamma)
- n\Gamma$.
\end{proposition}
\begin{corollary}\label{COR:N-GAMMA}
$(\Lambda - n\Gamma) + (\Lambda' - n' \Gamma)
\leq (\Lambda + \Lambda') - (n+n')\Gamma$.
\end{corollary}
\begin{corollary}\label{COR:D-K-M}
Let $d,k,m\in\Nat$, with $k < m$. Then,
\[
\bigl( \Lambda - (m + d)\Gamma \bigr) + (m\Gamma - k\Gamma)
\leq \Lambda - (d + k) \Gamma.\]
\end{corollary}
\begin{lemma}\label{LEM:D-K-M-I-J}
For every $i, j, k, m, d \in \Nats$ such that
$i, j, k < m$, and $i + j = m + d$,
\[
(\Lambda - i\Gamma) + (\Lambda' - j\Gamma) + (m\Gamma - k\Gamma)
\leq (\Lambda + \Lambda') - ( d + k) \Gamma.
\]
\end{lemma}
\begin{proof}
The lemma is a consequence of corollaries~\ref{COR:D-K-M}
and~\ref{COR:N-GAMMA}.
In fact,
\begin{multline*}
(\Lambda - i\Gamma) + (\Lambda' - j\Gamma) + (m\Gamma - k\Gamma) \leq\\
\leq \Pa{(\Lambda + \Lambda') - (m + d)\Gamma} + (m \Gamma - k \Gamma) \leq
(\Lambda + \Lambda') - (d + k) \Gamma.
\qedhere
\end{multline*}
\end{proof}
\sn
Set $\ig\Gamma := \Gamma - \Gamma$. It is straightforward to verify that
$\hat\Gamma - \hat\Gamma = \hat\Gamma = \hat\Gamma + \hat\Gamma$, and to
establish the following;
\begin{remark}\label{REM:SUM-ITERATE}
If $\Gamma$ is of the form $\gamma + \ig\Gamma$,
then $m\Gamma - k\Gamma = (m-k)\Gamma$ for every $k < m \in \Nat$.
\end{remark}
\subsection{Valued Fields}  \label{valf}
We need to recall some facts about valued fields. (Cf.\ \cite{E-P},  \cite {E},
\cite{Ri}).\sn
Let $\KF$ be a field, $G$ an ordered Abelian
group and
$\infty$ an element greater than every element of $G$.
A surjective map
$v:\> K\rightarrow G\cup\{\infty\}$
is a {\bf valuation}
on $\KF$ if for all $a,b\in \KF$:
(i) \ $v(a)=\infty$ if and only if $a=0$,\ \
(ii) \ $v(ab)= v(a)+v(b)$,\ \
(iii) \ $v(a-b)\geq \min\{v(a),v(b)\}$.
We say that $(\KF,v)$ is a {\bf valued field}, and shall write just
$\KF$ whenever the context is clear. It follows that
$v(a-b)= \min\{v(a),v(b)\}$ if $v(a)\not= v(b)$. The {\bf value
group} of
$\KF$ is
$v(\KF):=G$. The {\bf valuation ring} of $v$ is
$\vring := \{a\>;\> a\in \KF \mbox{ and } v(a)\geq 0\}$
and the {\bf valuation ideal} is
$\videal := \{a\>;\> a\in \KF \mbox{ and } v(a)> 0\}\;.$
The field $\cal O/\cal M$, denoted by $\kf$, is the {\bf residue
field}.
For $b\in \cal O$, $\overline{b}$ is its image under the
residue map.
\sn
A valued field $\KF$ is {\bf Henselian} if it satisfies Hensel's Lemma:
given a polynomial $p(x) \in \vring [x]$, and $a \in \kf$ a simple root
of the reduced polynomial $\overline{p(x)}\in \kf[x]$, we can find a
root $b\in \KF$ of $p(x)$ such that $\overline{b}=a$.
\sn
Let $\KF$ be a valued field, with value group $\vg$ and residue field~$\kf$,
with the same characteristic as~$\Kf$. Let $\vring$ be the valuation
ring of $\KF$, and $\videal$ its maximal ideal. A \intro{value group
section} is a map $t:\vg\to\KS$ such that $\forall \gamma\in\vg \quad
\val(t^\gamma)=\gamma \mbox{ and } t^{-\gamma} = 1/{t^\gamma}.$ Note
that $t$ must satisfy $t^0 = 1$. Additional conditions on $t$ might be
imposed later. A \intro{residue field section} is an embedding
$\iota:\kf\to\KF$ such that $\res{\iota x} = x$ for every~$x\in\kf$.
Whenever the context is clear, we will just write section to refer to
either a value group section or a residue field section. Further, we
will assume that $\iota$ is the identity. We recall the definition of a
generalized power series fields with factor set.

\begin{definizione}\label{factorset}
Let $(A, +, 0)$ and $(B, \cdot\, , 1)$ be two Abelian groups.
Then a $2$ \intro{co-cycle} is a map
$\factor: A\times A \to B$
satisfying the following conditions:
\begin{enumerate}
\item $\displaystyle\factor[\alpha,\beta]=\factor[\beta,\alpha]$.
\item $\displaystyle\factor[0,0]=\factor[0,\alpha]=\factor[\alpha,0]=1$.
\item $\displaystyle\factor[\alpha,\beta+\gamma]\factor[\beta,\gamma]=
\factor[\alpha+\beta,\gamma]\factor[\alpha,\beta]$.
\item $\displaystyle\factor[-\alpha,\alpha]=1$.
\end{enumerate}
\end{definizione}
The following is easily verified:
\begin{lemma}\label{LEM:CO-BOUNDARY}
Given a value group section $t$, the map $\factor:G\times G\to\kf^*$
defined by
\[
\factor[\alpha,\beta] := \frac{t^{\alpha}\,
t^{\beta}}{t^{(\alpha+\beta)}}
\]
is a 2 co-cycle.
Moreover, $t$ is a group homomorphism if and only if $\factor = 1$.
\end{lemma}
\sn
The co-cycle obtained from the section $t$ in
Lemma~\ref{LEM:CO-BOUNDARY} is denoted by
$\factor:={\bf d} t$.
\begin{definizione}[Factor set]
Let $\KF$ be a valued field containing its residue field $\kf$.
A \intro{factor set} is a 2 co-cycle
$\factor: G \times G \to \kf^*$.
\sn
If $t:G\to\KS$ is a section and $\factor={\bf d} t$ the corresponding
factor set, we will say that $t$ is a section with factor set $\factor$.
\end{definizione}
\begin{definizione}\label{psfwfactorset}
Let $G$ be an ordered Abelian group, and $\kf$ a field be given.
Given a 2 co-cycle $\factor:G\times G \to \kf^*$,
the field of generalized power series $\kGf$ with factor set $\factor$
is the set of formal series
$s = \sum_{\gamma\in G}a_{\gamma} t^\gamma$,
with $a_\gamma\in\kf$, whose support $\supp s := \{\gamma\>;\>
\gamma\in G
\mbox{ and } a_\gamma\not= 0\}$ is a well-ordered subset of $G$. Sum and
multiplication are defined formally, with the condition
\[
t^{\alpha}t^{\beta} = \factor[\alpha,\beta]t^{\alpha+\beta}.
\]
It is well-known that $\kGf$ is a valued field, with valuation given by
$v(s):= \min \supp s$
(by convention set $\min \supp
s=\infty$ if $\supp s= \emptyset$), value group $G$,
residue field $\kf$ and canonical section
$t(\gamma):=t^\gamma$.
With this definition, $t$ is a section with factor set $\factor$, and
$\kf(G, f)$ is the subfield of $\kf((G, f))$ generated by
$\kf \cup \set{t^\gamma: \gamma\in G}.$
If $f=1$ we denote $\kf((G, f))$ by $\kf((G))$.
\end{definizione}
\sn
A subfield $F$ of $\kf((G, f))$ is {\bf truncation closed} if whenever
$s=\sum_{\gamma\in G}a_\gamma t^\gamma\in
F$ and $g\in G$, the restriction $s_{<g}=\sum_{\gamma\in G^{<g}}s_\gamma
t^\gamma$ of $s$ to the initial segment $G^{<g}$ of $G$ also
belongs to $F$. Given a valued field
 $\KF$ with residue field $\kf$ and value group $G$ with $\kf(G, f)
\subset
\KF$,
 a {\bf truncation closed embedding} of $\KF$ in $\kf((G, f))$ over
$\kf(G, f)$,
is an embedding $\varphi$ such that $\varphi$ is the identity on
$\kf(G, f)$ and
$F:=\varphi
(\KF)$ is truncation
closed. Note that since the restriction of $\varphi$ to $\kf(G)$ is the
identity, $\varphi$ is in particular an embedding of $\kf$-vector
spaces.

\section{Tower of complements}\label{SEC:FAMILY}
From now on, we shall assume that $\KF$ is a valued field (with same
characteristic as its residue field) which
admits a value group section and a residue field section. Fix once and
for all a residue field section $\iota$
(we will assume that $\iota$ is the identity),
and a value group section $t:\vg\to\KS$.
\sn
From now on, we fix
$\bR\supseteq \vring$ a $\kf$-subalgebra of $\KF$ containing $\kf$ and the
image of~$t$.
\sn
Given $\Lambda\in\vgh$, define
\[\vring[\Lambda] := \set{x\in \bR: \val(x)\in\Lambda^R} =
\set{x \in \bR: \val(x)>\Lambda}.\]
Note that these are precisely the (possibly fractional) ideals of the
valuation ring.

\begin{remark}
\begin{enumerate}
\item $\vring[\Lambda]$ is a $\kf$-linear subspace of $\bR$;
\item for every $\gamma\in\Lambda^R$, $t^\gamma\in \vring[\Lambda]$, in
particular, $t^\gamma\in \vring[\gamma ^-]$;
\item $\vring[0^-] = \vring$;
\item $\vring[0^+] = \videal$;
\item $\vring[\gamma^{\pm}] = t^\gamma\vring[0^\pm]$;
\item\label{EN:RING-INCREASING} $\Gamma\leq\Lambda$ if and only if
$\vring[\Gamma]\supseteq\vring[\Lambda]$;
\item\label{EN:RING-CONTINUOUS} if $\Lambda^R$ has no minimum,
then $\vring[\Lambda] = \bigcup_{\gamma > \Lambda}
\vring[\gamma^{\pm}]=\bigcap_{\gamma <\Lambda}\vring[\gamma^{+}]$  ;
\item $\vring[\Lambda \ar \Gamma] = \vring[\Lambda]\vring[\Gamma]$;
\item $\vring[\gamma + \Lambda] = t^\gamma \vring[\Lambda]$.
\end{enumerate}
\end{remark}
\begin{definizione}\label{wtoc}
A ($t$-compatible) \intro{ weak tower of complements} (of the valuation
ring
$\vring$) for $\bR$
%
%
is a family $\vcompf = \set{\vcomp[\Lambda]:\Lambda\in \vgh}$ of subsets
of $\bR$ indexed by $\vgh$, such that:
\begin{enumerate}
\renewcommand{\theenumi}{{\textbf{C}}\Alph{enumi}}
\item\label{AX:CA} $\vcomp[\Lambda]$ is a $\kf$-subspace of $\bR$;
\item\label{AX:CB} $\vcomp[\Lambda] \oplus \vring[\Lambda] = \bR$,
as $\kf$-spaces;
\item\label{AX:CC} $t^\gamma \kf \subseteq \vcomp[\gamma^+]$;
\item\label{AX:CD} $\Gamma\leq\Lambda$ if and only if $\vcomp[\Gamma]
\subseteq\vcomp
[\Lambda]$;
\setcounter{enumi}{5}
\item\label{AX:CF} $\vcomp[\gamma + \Lambda] = t^\gamma \vcomp[\Lambda]$
\end{enumerate}
In particular, $\vcomp := \vcomp[0^-]$ is a $\kf$-complement
of~$\vring$. Note that $\vcomp[+\infty] = \bR$, and $\vcomp[-\infty] =
\{0\}$.\mn
A ($t$-compatible) \intro{tower of complements} is a ($t$-compatible)
weak tower of complements
satisfying the following additional axiom, instead of Axiom~\ref{AX:CF}:
\begin{enumerate}
\renewcommand{\theenumi}{{\textbf{C}}\Alph{enumi}}
\setcounter{enumi}{4}
\item\label{AX:CE} $\vcomp[\Lambda + \Gamma] \supseteq \vcomp[\Lambda]
\vcomp[\Gamma]$.
\end{enumerate}
We then also say that {\bf $\vcompf$ is multiplicative}. In this case,
$\vcomp$ is a $\kf$-algebra complement.
%
%
\end{definizione}

\begin{remark}
Any tower of complements is a weak tower of complements.
\footnote{Example~\ref{EX:WTOC} produces a W.T.o.C. that is not a T.o.C..}
\end{remark}
\begin{proof}
We have to prove that if $\vcompf$ is a tower of complements, then
$t^\gamma\vcomp[\Lambda] = \vcomp[\gamma + \Lambda]$.
Since
$t^\gamma \in \vcomp[\gp]$, we have $t^\gamma \vcomp[\Lambda]\subseteq
\vcomp[\gp + \Lambda] = \vcomp[\gamma + \Lambda]$. Conversely,
using the above inclusion we get:
\sn
$
\vcomp[\gamma + \Lambda] = t^{\gamma} t^{-\gamma} \vcomp[\gamma + \Lambda]
\subseteq t^\gamma \vcomp[-\gamma + \gamma + \Lambda] =
t^\gamma \vcomp[\Lambda].$
\qedhere
\end{proof}
\mn
For the remainder of this section, we will assume that $\vcompf$ is a
weak tower of complements for~$\bR$, unless we explicitly specify
otherwise. In the following Lemma, we list useful properties of weak
towers of complements.
\begin{lemma}\label{LEM:COMPL-BASIC}
\begin{enumerate}
\makeatletter
\renewcommand{\p@enumi}{\ref{LEM:COMPL-BASIC}--}
\makeatother
\item $\kf \subseteq \vcomp[0^+] $.
\item $\vcomp[\gamma^-] = t^\gamma \vcomp$.
\item If $\Gamma \leq \Lambda$, then
\[\begin{array}{r@{\extracolsep{0.9ex}\;\; =\ }c@{}c@{}c@{}c@{}cr}
      \bR & \vcompG &\oplus& \Pa{\vringG \cap \vcompL} &\oplus&
\vringL,& \\
\vcompL & \vcompG &\oplus& \Pa{\vringG \cap \vcompL}, &&& \text{and}\\
\vringG & &&\Pa{\vringG \cap \vcompL} &\oplus& \vringL.
\end{array}\]
\item\label{EN:COMPL-BASIC-DEC} Let $\Gamma \leq \Lambda$. Let $x = x_1
+ x_2$, with
$x_1 \in \vcompG$ and $x_2 \in \vring[\Gamma]$;
if $x \in \vcomp[\Lambda]$, then $x_i \in \vcompL$; if $x \in \vringL$,
then $x_i \in \vringL$.
\item $\displaystyle \vring[\gamma^-] \cap \vcomp[\gamma^+] = t^\gamma\kf$.
\item $\displaystyle \vcomp[\gamma^+] = \vcomp[\gamma^-] \oplus t^\gamma\kf$.
\item If $\Lambda^R$ has no minimum, then
$\displaystyle \vcomp[\Lambda] = \bigcap_{\gamma\in\Lambda^R}
\vcomp[\gamma^\pm]$.
\end{enumerate}
\end{lemma}
\begin{proof}
1.\ \  This follows immediately from $t^0=1$ and Axiom~\ref{AX:CC}.
\mn
2. \ \
We have $t^\gamma\in\vcomp[\gp]$, so $t^\gamma\vcomp \subseteq
\vcomp[\gp] \vcomp\subseteq \vcomp[\gp+0^-] = \vcomp[\gm]$. Conversely,
let $x\in\vcomp[\gm]$, $y := t^{-\gamma}x$. Decompose $y = y_1+y_2$,
$y_1\in\vcomp$, $y_2\in\vring$. Hence, $x = t^\gamma y = t^\gamma y_1 +
t^\gamma y_2$. We have that $t^\gamma y_1 \in \vcomp[\gm]$ by the
previous point. Moreover, $\val(t^\gamma y_2) = \gamma + \val(y_2) \geq
\gamma$, therefore $t^\gamma y_2 \in\vring[\gm]$. However,
$x\in\vcomp[\gm]$, hence $t^\gamma y_2 \in \vring[\gm]\cap\vcomp[\gm]
=(0)$, so $x=t^\gamma y_1 \in t^\gamma\vcomp$.
\mn
3.\ \
Let $x \in \bR$. Write $x = x_1 + y$, where $x_1 \in \vcomp[\Gamma]$ and
$y \in \vring[\Gamma]$. Write $y = x_2 + x_3$, where $x_2 \in
\vcomp[\Lambda]$ and $x_3 \in \vring[\Lambda]$. Note that $x_3 \in
\vring[\Gamma]$, hence $x_2 = y - x_3 \in \vring[\Gamma]$.
Therefore, we have decomposed $x$ as $x = x_1 + x_2 + x_3$.
\sn
To prove the uniqueness of the decomposition,
assume that $x = x_1' + x_2' + x_3'$,
with $x_1' \in \vcomp[\Gamma]$, $x_2' \in \vring[\Gamma]\cap\vcomp[\Lambda]$,
and $x_3' \in \vring[\Lambda]$.
Define $y' := x_2' + x_3'$.
Note that $x_3' \in \vring[\Gamma]$, and therefore $y \in \vring[\Gamma]$.
By Axiom~\ref{AX:CB}, $x_1' = x_1$, and therefore $y' = y$.
Again by Axiom~\ref{AX:CB}, $x_2' = x_2$ and $x_3' = x_3$.
\sn The other assertions follow by similar considerations.
\mn
4.\ \
Immediate from the previous point.
\mn
5.\ \ It is immediate by \ref{AX:CC} that $\vring[\gamma^-] \cap
\vcomp[\gamma^+] \supset t^\gamma\kf$.
Conversely,
let $x\in\vring[\gm]\cap\vcomp[\gp]$, $x\neq 0$.
Hence, $\val(x) = \gamma$ by \ref{AX:CB}.
Write $x$ as $ct^\gamma + z$, with $c\in\ks$, and $z\in\vring[\gp]$.
By \ref{AX:CC}, $ct^\gamma\in\vcomp[\gp]$,
hence $z\in\vcomp[\gp]\cap\vring[\gp]$, therefore $z=0$.
\mn
6.\ \
$\vcomp[\gp] = \vcomp[\gm] \oplus \Pa{\vcomp[\gp] \cap \vring[\gm]}
= \vcomp[\gm] \oplus t^\gamma\kf$.
\mn
7.\ \
By \ref{AX:CD} we have
$\displaystyle \vcomp[\Lambda] \subseteq \bigcap_{\gamma\in\Lambda^R}
\vcomp[\gamma^\pm]$.
Conversely, since $\Lambda^R$ has no minimum,
$\bigcap_{\gamma\in\Lambda^R}\vcomp[\gp]
= \bigcap_{\gamma\in\Lambda^R}\vcomp[\gm]$.
Let $x \in \bigcap_{\gamma\in\Lambda^R}\vcomp[\gp]$.
Decompose $x = x_1 + x_2$, $x_1\in\vcomp[\Lambda]$, $x_2\in\vring[\Lambda]$.
By the previous inclusion, $x_1\in\bigcap_{\gamma\in\Lambda^R}\vcomp[\gm]$.
Assume for a contradiction that $x_2 \neq 0$. Then $\gamma := \val(x_2)
>
\Lambda$. Since $\Lambda^R$ has no minimum, $\Lambda < \gm$,
therefore $x \in \vcomp[\gm]$, and thus $x_2 \in \vcomp[\gm]$.
Hence, $x_2\in\vcomp[\gm] \cap \vring[\gm] = (0)$, a contradiction.
\qedhere
\end{proof}
\begin{corollary}\label{corto7}
The weak tower of complements $\vcompf$ is uniquely determined by
 $\vcomp[0^-] \in \vcompf$.
\end{corollary}
\begin{proof}
 First observe that by 2. and 6. of Lemma \ref{LEM:COMPL-BASIC}
$\vcomp[0^-]$ uniquely determines
$\vcomp[\gamma^\pm]$ for all $\gamma \in G$. Second, if $\Gamma$ is a
cut of $G$ not of the form $\gamma^\pm$, then
$\vcomp[\Gamma]$ is the intersection of all $\vcomp[\gamma^\pm]$ for
$\Gamma
< \gamma \in G$ (Lemma
 \ref{LEM:COMPL-BASIC} 7.).
\end{proof}
\begin{corollary}
$t$ is a section with factor set (in~$\kf^*$).
\end{corollary}
\begin{proof} We compute:
\[
d:=t^\alpha t^\beta t^{-\alpha -\beta}
\in t^\alpha t^\beta t^{-\alpha - \beta}\vcomp[0^+]=
\vcomp[\alpha + \beta + (-\alpha-\beta) + 0^+] =\vcomp[0^+].
\]
Hence, $d \in \vring \cap \vcomp[0^+] = \kf$.
\end{proof}
\begin{mainexample}
Let $\factor: \vg\times\vg \to \kf^\star$ be a factor set.
For every $\Lambda\in\vgh$, define $\kf[[\Lambda^L,\factor]]$ as the subset
of $\kGf$ of the power series with support contained in $\Lambda^L$.
The family $\set{\kf[[\Lambda^L,\factor]]: \Lambda\in\vgh}$ is a tower of
complements for $\kGf$.\\
Moreover, if $t$ has factor set $\factor$, and $\phi$ is a truncation-closed
 embedding of $\KF$ in $\kGf$, preserving the section $t$ and the embedding of the residue field $\kf$, then the family
\[
\vcompf := \set{ \phi^{-1}\Pa{\kf[[\Lambda^L,\factor]]}: \Lambda\in\vgh}
\]
is a tower of complements for $\KF$.
\end{mainexample}
\begin{remark}
Note that in this case
$
\bigcup_{\substack{\Lambda \in \vgh \\ \Lambda < + \infty}}
\vcompL \subsetneq \KF.
$
Since $\KF= A[+\infty]$, we therefore have at least one cut ($\Gamma =
+\infty$) such that $A[\Gamma]$ is not equal to the union of
$A[\Lambda]$, for $\Lambda <\Gamma$, because the image of $\KF$ will
always contain power series with unbounded support. This shows that
$\vcompf$ is {\it not} a ``continuous family'': if $\Gamma$ is a
supremum of an increasing sequence of cuts $\Lambda_i$, then $A[\Gamma]$
is {\it not necessarily} the union of the $A[\Lambda_i]$'s. Note that if
instead $\Gamma$ is an infimum of a decreasing sequence of cuts
$\Lambda_i$, then $A[\Gamma]$ is the intersection of the
$A[\Lambda_i]$'s (Lemma \ref{LEM:COMPL-BASIC} 7.).
\end{remark}
\subsection{Embeddings in power series}     \label{EPS}
\begin{definizione}[Embedding $\fsum$]
For every $\gamma\in\vg$, we can write
\[
\bR = \vcomp[\gm] \oplus t^\gamma \kf \oplus \vring[\gp] =
t^\gamma \vcomp \oplus t^\gamma \kf \oplus t^\gamma\videal.
\]
Given $x \in \bR$, decompose $x = x_1 + a_\gamma t^\gamma + x_3$,
where $x_1\in\vcomp[\gm]$ and $x_3\in\vring[\gp]$.
Consider the formal sum $\xs := \sum_{\gamma\in\vg} a_\gamma t^\gamma$.
$\fsum$~defines a map from $\bR$ to $\KG$, the set of maps from
$\vg$ to~$\kf$.
Note that $\KG$ is an Abelian group under point-wise addition; it
is obvious that $\fsum$ is a homomorphism of (additive) groups. Define
also $\supp x := \supp \xs$.

Observe that the definitions of $\xs$ and $\supp$ depend on the given
tower of complements ${\cal A}$.
\end{definizione}
\begin{lemma}
$\fsum$ is injective.
Moreover, $\val(x) = \min(\supp x)$.
\end{lemma}
\begin{proof}
Let $x \in \bR$, and $\gamma := \val(x)$.
Note that if $\gamma < \infty$, then $a_\gamma \neq 0$.
Therefore, $\gamma \in \supp x$.
Hence, if $\xs = 0$, then $\gamma = \infty$, that is, $x = 0$, and
$\fsum$ is injective.
Conversely, let $\lambda < \gamma$, and assume for contradiction that
$a_\lambda \neq 0$.
Then, $x = x_1 + a_\lambda t^\lambda + x_3$, where $x_1\in\vcomp[\lm]$
and $x_3\in\vring[\lp]$.
Therefore, $\gamma=\val(x)= \min\{v(x_1), \lambda, v(x_3)\} \leq
\lambda$, a contradiction.
\end{proof}
\sn
We now prove that the image of $\fsum$ is contained in~$\kG$.
\begin{proposition}
If $x\in \bR$, then $\supp \xs$ is well-ordered.
\end{proposition}
\begin{proof}
Suppose not.
Let $\alpha := \val x$.
Let $\gamma_1 > \gamma_2 > \ldots \in \supp \xs$.
Since $\supp \xs$ is bounded below by $\alpha$, and $\vgh$ is complete,
there exists
\[
\Lambda := \inf \set{\gamma_i^+: i\in\Nat} = \inf\set{\gamma_i^-: i \in\Nat}
\ = \set{\gamma_i: i\in\Nat}^- \in \vgh.\]
Moreover, $\Lambda^R$ has no minimum, and $\alpha < \Lambda$.
Decompose $x = x_1+x_2$, $x_1\in\vcomp[\Lambda]$, $x_2\in\vring[\Lambda]$.
Let $\beta:=\val(x_2) > \alpha$, and let $\gamma:=\gamma_i$ such that
$\gamma < \beta$ (it exists, because $\Lambda^R$ has no minimum).
Write $x = y_1 + a_{\gamma} t^{\gamma} + y_2$, $y_1\in\vcomp[\gm]$,
$\val(y_2)>\gamma$, and $a_{\gamma} \neq 0$.
However, $x_1 \in \vcomp[\Lambda] \subseteq \vcomp[\gm]$, and $x_2 \in
\vring[\beta^-]\subseteq\vring[\gm]$, therefore, by~\ref{AX:CB},
$x_1 = y_1$, and $x_2 = a_{\gamma} t^{\gamma} + y_2$, a contradiction
to $\val(x_2) = \beta > \gamma$.
\end{proof}
\mn
Since the minimum of $\supp\xs$ is exactly $\val x$, we have that:
\begin{corollary}
The map  $\fsum$ is
an embedding of (additive) valued groups.
\end{corollary}
\mn
For the rest of this section, we will assume that $\vcompf$ is a
(weak) tower of complements for~$\bR$.
\begin{definizione}\label{MU}
Set:
$\mu(x) := (\supp x)^+\in\vgh.$
\end{definizione}
\n
Note that $\mu(x)=-\infty$ if and only if $x=0$.
\begin{lemma}\label{LEM:MU}
Let $\Lambda\in\vgh$ and $x,y \in \bR$.
\begin{enumerate}
\makeatletter
\renewcommand{\p@enumi}{\ref{LEM:MU}--}
\makeatother
\item\label{EN:MU-1}
If $\mu(x) = \Lambda$, then $x \in \vcomp[\Lambda]$.
\item\label{EN:MU-2} In general, $x \in \vcompL$ if and only if $\Lambda \geq \mu x$.
Equivalently, $\vcompL = \set{ x: \mu x \leq \Lambda}$.
\item $\mu x =\inf \set{ \Lambda: x \in \vcompL}$.
\item\label{EN:MU-SUM} $\mu(x + y) \leq \max (\mu x, \mu y)$.
\item\label{EN:MU-MULT} If $\vcompf$ is a \emph{tower of complements},
then $\mu(x y)
\leq \mu x + \mu y$.
\end{enumerate}
\end{lemma}
\begin{proof}
1. \ \
Decompose $x = x_1+ x_2$, $x_1\in\vcomp[\Lambda]$, $x_2\in\vring[\Lambda]$.
If, for contradiction, $x_2\neq 0$, then $\val(x_2) \in \Lambda^R$.
However, $\val(x_2) \in \supp x$, a contradiction.
\mn
2.\ \ $\Leftarrow$) By~\ref{EN:MU-1} and~\ref{AX:CD}.\\
$\Rightarrow$) If, for contradiction, $x\in\vcomp[\Lambda]$ and $\Lambda < \mu x$,
then there exists $\gamma\in\supp x$ such that $\gamma>\Lambda$.
Decompose $x = x_1 + c_\gamma t^\gamma + x_2$.
Decompose $x_1 = z_1 + z_2$, $z_1\in\vcomp[\Lambda]$, $z_2 \in \vring[\Lambda] \cap \vcomp[\gm]$.
Hence, $x = z_1 + z_2 + c_\gamma t^\gamma + x_2$.
Since $x\in\vcomp[\Lambda]$, $z_2 + c_\gamma t^\gamma + x_2 = 0$.
Thus, $z_2 \in\vcomp[\gm]\cap\vring[\gm] = (0)$, so $c_\gamma = 0$,
contradicting the fact that
$\gamma\in\supp x$.
\mn
3.\ \ This is a rewording of \ref{EN:MU-2}.
\mn
4.\ \ This follows from Axiom~\ref{AX:CA}, plus \ref{EN:MU-2}.
\mn
5.\ \
If either $x=0$ or $y=0$, the conclusion is trivial.
By~\ref{EN:MU-1}, $x\in\vcomp[\mu x]$, $y\in\vcomp[\mu y]$,
therefore $xy \in \vcomp[\mu x + \mu y]$ by \ref{AX:CE}.
This, together with \ref{EN:MU-2}, implies that $\mu(xy) \leq \mu x + \mu y$.
\qedhere
\end{proof}
Assume now that $\bR=\KF$ and that $\vcompf$ is multiplicative.
Let $\factor$ be the factor set of\/ $t$. We shall show that $\fsum$ is
an embedding of valued fields in the
power series field, with multiplication twisted by~$\factor$.
\begin{thm}\label{THM:EMBEDDING}
For every $x,y \in \KF$, $\fsum(xy) = (\fsum x)(\fsum y)$ (with the
multiplication of $\kGf$). Therefore, $\fsum$ is a truncation-closed
embedding of valued fields of $\KF$ in $\kGf$.
\end{thm}
\begin{proof}
If not, let $x,y\in\KF$ of minimal length%
\footnote{The length of $x$ is the order type of $\supp \fsum x$.}
such that $\fsum(xy) \neq (\fsum x)(\fsum y)$.
Let $z := \fsum x \fsum y$, and $\gamma\in\vg$ minimal with
$\fsum(xy)(\gamma) \neq z (\gamma)$.%
\footnote{We are using the notation $(\sum a_\lambda t^\lambda)(\gamma)
:= a_\gamma$.}
Since $\mu (xy) \leq \mu x + \mu y$ and $\mu(\fsum x \fsum y) \leq \mu
(\fsum x) + \mu (\fsum y) = \mu x + \mu y$, $\gamma < \mu x + \mu y$.
Let $\alpha,\beta\in\vg$ such that $\alpha < \mu x$, $\beta < \mu y$ and
$\alpha + \beta = \gamma$.
Decompose $x = x_1 + a_\alpha t^\alpha + x_2$, $y
= y_1 + b_\beta t^\beta + y_2$.
Thus,
\begin{eqnarray*}
x y & = & (x_1 + a_\alpha t^\alpha + x_2)(y_1 + b_\beta t^\beta + y_2)\\
 & = &
x_1 y + (a_\alpha t^\alpha + x_2) y_1 + \factor[\alpha,\beta]
a_\alpha b_\beta t^{\gamma} + \smallo{t^{\gamma}}
\footnote{Here $\smallo{t^{\gamma}}$ denotes an element of value
$>\gamma$.}.
\end{eqnarray*}
Since $\alpha < \mu x$, $a_\alpha t^\alpha + x_2 \neq 0$, and similarly
$b_\beta t^\beta + y_2 \neq 0$.
By minimality of $x$ and $y$, $\fsum(x_1 y) = \fsum x_1 \fsum y$ and
$\fsum\Pa{(a_\alpha t^\alpha + x_2) y_1} = \fsum(a_\alpha t^\alpha + x_2)
\fsum y_1$.
Moreover, $\fsum\Pa{\factor[\alpha,\beta] a_\alpha b_\beta t^{\gamma}} =
\factor[\alpha,\beta] a_\alpha b_\beta t^{\gamma}$
by definition of $\fsum$, and $\fsum\Pa{\smallo{t^{\gamma}}} =
\smallo{t^{\gamma}}$.
Therefore,
\begin{eqnarray*}
\fsum(x y) & = & \fsum(x_1 y) + \fsum\Pa{(a_\alpha t^\alpha + x_2) y_1} +
\fsum\Pa{\factor[\alpha,\beta] a_\alpha b_\beta t^{\gamma}} +
\fsum(\smallo{t^{\gamma}})\\
 & = & \fsum x_1 \fsum y + \fsum(a_\alpha t^\alpha + x_2) \fsum y_1
+ a_\alpha t^\alpha b_\beta t^{\beta} + \smallo{t^{\gamma}} \\
 & = & \fsum x \fsum y - a_\alpha t^\alpha \fsum y_2 - \fsum x_2
(b_\beta t^\beta + \fsum y_2) + \smallo{t^{\gamma}}.
\end{eqnarray*}
Thus, $z - \fsum(x y) = \smallo{t^{\gamma}}$, a contradiction.
\end{proof}
\begin{corollary}\label{COR:EMBEDDING}
Let $\factor$ be the factor set of $t$. There is a one-to-one
correspondence between towers of complements for $\KF$ and
truncation-closed embeddings of $\KF$ in~$\kGf$.
\end{corollary}
\begin{proof}
Consider the maps $\phi\mapsto A_\phi$ that maps a truncation-closed
embedding $\phi$ to the corresponding T.o.C. (induced by $\phi$), and
${\cal A}\mapsto {\cal S}_{\cal A}$ that maps a T.o.C.\ ${\cal A}$ to
the truncation-closed embedding induced by it. We want to prove that the
two maps are inverses of each other.
We know that a T.o.C.\ ${\cal A}$ is uniquely determined by the
corresponding function $\mu_{\cal A}$. Let ${\cal B}$ be a T.o.C.
Let us prove that ${\cal A}_{{\cal S}_{\cal B}} = {\cal B}$. We have
that
\[
x\in B[\Gamma]\Longleftrightarrow {\cal S}_{\cal B}(x)\in k[[\Gamma^L,
f]] \Longleftrightarrow x \in {\cal A}_{{\cal S}_{\cal B}}[\Gamma]\>.
\]
Conversely, let $\phi$ be a truncation-closed embedding. Let us prove
that $\alpha:={\cal S}_{A_\phi}=\phi$. Let ${\cal B}:={\cal A}_\phi$.
Assume, for a contradiction, that $\phi$ is not equal to $\alpha$, and
choose $x \in K$ of minimal length (w.r.t.\ the T.o.C.\ ${\cal B}$) such
that $\phi(x)$ is not equal to $\alpha(x)$. Note that
\begin{equation}                            \label{star}
\phi(x) \in k[[\Gamma^L, f]]\Longleftrightarrow x \in B[\Gamma]
\Longleftrightarrow \alpha(x) \in k[[\Gamma^L, f]]\>.
\end{equation}
Let $\Gamma < \mu_{\cal B}(x)$; split $x = x_1 + x_2$ at $\Gamma$. By
the minimality of $x$, $\phi(x_1) = \alpha(x_1)$. Therefore, $v(\phi(x)
- \alpha(x)) > \Gamma$ for every $\Gamma < \mu_B(x)$, and thus
$v(\phi(x) - \alpha(x)) > \mu(x)$, contradicting (\ref{star}).
\end{proof}

\section{Truncation $\kf$-algebra complements of~$\vring$}
\label{oldopenquestion}
Fix a residue field section $\iota: \kf \to \KF$, which we
assume to be the inclusion map.
Let $t: \KF^* \to G$ be a value group section, with factor set~$f$.
A $\kf$-complement $A$ of~$\vring$ is {\bf compatible with $t$} if
$t^\gamma \in A$ for every $0 > \gamma \in G$.
\sn
Recall that we call a $\kf$-algebra complement $A$ of~$\vring$ a
truncation $\kf$-algebra complement if there is a
truncation-closed embedding $\phi: A
\to \kf ((G, f'))$ (preserving $\iota$)
for some factor set~$f'$, such that $A = \phi^{-1}\Pa{\kf ((G^{<0},
f'))}$. We say that $A$ is a {\bf truncation $\kf$-algebra complement
compatible with $t$} if moreover $\phi(t^\gamma) = t^\gamma$ for every
$\gamma \in G$, and $f = f'$.
\sn
It follows from Corollaries \ref{corto7} and \ref{COR:EMBEDDING}
that $A$~is a truncation $\kf$-algebra complement compatible with~$t$ if
and only if $A = \vcomp[0^-]$, for some \ToC $\vcompf$ compatible
with~$t$.
\sn
Our aim in this section is to find a valued field~$\KF$, with a residue
field section~$\iota$
and a value group section~$t$, and a $\kf$-algebra complement~$A$
compatible with~$t$, such that $A$ is
not a truncation $\kf$-algebra complement compatible with~$t$.
We will leave open the question whether $A$ might be a truncation
$\kf$-algebra complement compatible with a different value group
section~$t'$.
\mn
\begin{definition}
Let $\vcompf := \Pa{A[\Gamma]}_{\Gamma \in \vgh}$ be a family of subsets of~$\KF$, indexed by~$\vgh$.
$\vcompf$~is a $t$-compatible {\bf candidate weak tower of complements}
(\CWToC for short)
if it satisfies the axioms CA, CC, CD, CF, and instead of CB
the following axioms:
\begin{enumerate}
\item[\textbf CB1.] $A[\Lambda] \cap \vring[\Lambda] = (0)$;
\item[\textbf CB2.] $A[0^-] + \vring[0^-] = \KF$;
\item[\textbf CG.] $A[\Gamma] = \bigcup_{\lambda > \Gamma} A[\lambda^-]$.
\end{enumerate}
If in addition $\vcompf$ satisfies the axiom CE,
then we say that $\vcompf$ is multiplicative, or that $\vcompf$ is a
$t$-compatible {\bf candidate tower of complements} (\CToC for short).
\end{definition}
\begin{remark}
Every \WToC is a \CWToC; every \ToC is a \CToC.
\end{remark}
\begin{lemma}
Let $\vcompf$ be a \CWToC.
Then,
\begin{enumerate}
\item $A[\gamma^-] = t^\gamma A[0^-]$;
\item
$A[\gamma^+] = A[\gamma^-]  + t^\gamma \kf = t^\gamma A[0^+]$;
\item $\KF = A[\gamma^-] \oplus \vring[\gamma^-] =
A[\gamma^+] \oplus \vring[\gamma^+] =
A[\gamma^-] \oplus t^\gamma \kf \oplus \vring[\gamma^+]$.
\end{enumerate}
\end{lemma}

Given a \CWToC~$\vcompf$, $x \in \KF$ and $\lambda \in G$, we can
decompose $x = x'_\lambda + a_\lambda(x) t^\lambda + x''_\lambda$
uniquely, in such a way that $x'_\lambda \in A[\lambda^-]$,
$a_\lambda(x) \in \kf$, and $v(x''_\lambda)> \lambda$. We define
$\fsum x := \sum_\lambda a_\lambda(x) t^\lambda$
%
and $\supp x := \supp \fsum x$ as in Section~\ref{EPS}.

\par\smallskip
The proofs of the following two lemmata are easy.

\begin{lemma}
Let $\vcompf$ be a \CWToC. The following are equivalent:
\begin{enumerate}
\item $\vcompf$ is a \WToC;
\item for every $\Gamma \in \vgh$, $A[\Gamma] + \vring[\Gamma] = \KF$;
\item for every $x \in \KF$, $\supp x$ is well-ordered;
\item for every $x \in \KF$, for every $\Gamma \in \vgh$, there exists $\bar\lambda > \Gamma$ such that, for every $\Gamma < \lambda < \bar \lambda$, $x'_\lambda = x'_{\bar \lambda}$;
\item for every $x \in \KF$, for every $\Gamma \in \vgh$, there exists $\bar\lambda > \Gamma$, such that $x'_{\bar \lambda} \in A[\Gamma]$.
\end{enumerate}
\end{lemma}

\begin{lemma}
Let $A^-$ be a $\kf$-complement of~$\vring$ compatible with~$t$.
Then there exists a unique \CWToC $\vcompf$ such that $A[0^-] = A^-$.
$\vcompf$~is multiplicative iff $A^-$ is multiplicative.\\
$\vcompf$~is defined in the following way:
\[
A[\Lambda] := \bigcap_{\gamma > \Lambda} t^\gamma A^-.
\]
\end{lemma}

Let $A^-$ be
a $\kf$-complement of~$\vring$ compatible with~$t$, and
the family $\vcompf$ defined in the lemma be the \CWToC induced
by~$A^-$.

If $A^-$ is a truncation $\kf$-algebra compatible with~$t$, then there
exist at least one \WToC $\vcompf'$ such that $\vcompf'[0^-] = A^-$.
Then by the above lemma, $\vcompf' = \vcompf$. Therefore, a
necessary and sufficient condition for $A^-$ to be a truncation
$\kf$-algebra compatible with~$t$ is that, for every $x \in \KF$,
$\supp x$ is well-ordered.

We will now define a valued field $\KF$, a $\kf$-complement $B^-$ of~$\vring$
(compatible with chosen residue field and value group sections),
and an element $d \in \KF$, such that $\supp d$ is not well-ordered.

Fix a field~$\ff$.
Let $\FF := \ff((\Zed))$, with the canonical inclusion $\iota: \ff \to \FF$ and
value group section $t: \Zed \to \FF$. Call $\rho: \vring \to \ff$ the residue map.
Let $A^- := \ff[[\Zed^{<0}]]$:
by definition, $A^-$ is a truncation $\ff$-algebra.
Define the maps $h', h'' : \FF \times \Zed \to \FF$ and $g : \FF \times \Zed \to \ff$,  $h'(x, \gamma) = x'_\gamma$, $g(x, \gamma) = a_\gamma(x)$, $h''(x,\gamma) = x''_\gamma$.
Let $c := \sum_{n \geq 0} t^n \in \FF$.
Consider the first-order structure, in the sorts $\FF, \ff, \Zed$,
\[
M := \Pa{
\FF, \ff, \Zed; A^-, +_\FF, \cdot_\FF, +_\Zed, \leq_\Zed, +_\ff, \cdot_\ff, v, t,
\rho, \iota, h', g, h'', c}.
\]
Let
\[
\tilde M = \Pa{\KF, \kf, G; B^-, +_\KF, \cdot_\KF, +_G, \leq_G, +_\kf,
\cdot_\kf, \tilde v, \tilde t, \tilde \rho, \tilde\iota, \tilde{h'},
\tilde g, \tilde{h''}, c}\]
be an $\omega$-saturated elementary extension of~$M$.
It is easy to see that $B^-$ is an $\ff$-complement to~$\vring$,
satisfying $B^- \cdot B^- \subseteq B^-$, and that
$\tilde{h'}(x, \gamma) = x'_\gamma$,
$\tilde g(x, \gamma) = a_\gamma(x)$,
$\tilde{h''}(x, \gamma) = x''_\gamma$.
Moreover, since $g(c, \gamma) = 1$ for every $\gamma \in \Zed^{\geq 0}$
holds in $M$, we must have that $\tilde g(c, \gamma) = 1$ for every
$\gamma \in G^{\geq 0}$. Thus, $\supp c = G^{\geq 0}$. However, since
$\tilde M$ is $\omega$-saturated, $G^{\geq 0}$ is not well-ordered.
Therefore, $B^-$ is not a truncation $\kf$-algebra compatible
with~$\tilde t$.

%
\section{Building a tower of complements} \label{building}
Let $\KF$, $\vg$ and $\kf$ be as in \S\ref{SEC:FAMILY}. To simplify the
notation, we will assume that the section $t: \vg \to \KS$ has trivial
factor set. The aim of this section is to build a tower of complements
for~$\KF$. Given $\bR$ a subring of $\KF$ containing $\vring$ and the
image of~$t$, we define a tower of complements for $\bR$ as a family
$\vcompf$ of subsets of $\bR$ satisfying the
axioms~\ref{AX:CA}--\ref{AX:CE}.

\subsection{The basic case}  \label{basicase}
First of all, consider $\kgring$, the subring of $\KF$ generated by the
monomials $ct^\gamma$. There is one and only one tower of complements
for~$\kgring$: for each~$\Lambda$, $\vcomp[\Lambda]$ is the $\kf$-vector
subspace of $\KF$ generated by the monomials $t^\gamma$ such that
$\gamma < \Lambda$.

\subsection{Extension to quotient fields}  \label{eqf}
Next, consider $\kg$, the field of quotients of $\kgring$.
There is one quick way of constructing a tower of complements for
$\kg$: notice that there exists a unique analytic embedding $\phi$ from
$\kg$ in the field of power series $\kG$ preserving $\kf$ and the
section $t$. Moreover, $\phi$ is truncation-closed, hence the family
\[
\vcompf := \set{\phi^{-1}(k[[\Lambda^L]]): \Lambda\in\vgh}
\]
is a tower of complements for $\kg$.

\par\medskip
However, as we intend to construct our complements intrinsically,
without the use of truncation-closed embeddings in power series fields,
we wish to give a general construction for the extension of towers of
complements from a ring to its quotient field.

\begin{definizione}
Given $\Lambda \in \vgh$, define $\Zed\Lambda := \sup \set{n
\Lambda: n\in\Zed} \in \vgh$.
\end{definizione}
\sn
Note that $\Zed\Lambda + \Zed\Lambda = \Zed\Lambda > 0$.

\begin{proposition}                        \label{quot}
Let $\bR$ a subring of $\,\KF$ containing $\vring$ and the image
of $\,t$. Suppose that $\vcompf$ is a tower of complements for $\bR$.
Define a tower $\vcompbf$ on the quotient field of $\bR$ by
\[
B[\Gamma] = \mathrm{span}_k \Bigl\{
\frac{r}{1 + a}: a, r \in R, v(a) > 0, \mu(r) + \Zed \mu(a) \leq \Gamma
\Bigr\}
\]
for each $\Gamma \in \vgh$. Then $\vcompbf$ is a tower of complements
for the quotient field of~$\bR$.
\end{proposition}
\begin{proof}
Property \ref{AX:CA} for $\vcompbf$ holds by definition. Properties \ref{AX:CC} and \ref{AX:CD} for
$\vcompbf$ are directly inherited from the corresponding properties of~$\vcompf$.
%


\par\smallskip
We show that $\vcompbf$ has property~\ref{AX:CE}.
Take cuts $\Lambda,\Gamma$ and elements $a,a',r,r' \in R$ with
$v(a) > 0$, $v(a') > 0$, $\mu(r) + \Zed \mu(a) \leq \Lambda$ and
$\mu(r') + \Zed \mu(a') \leq \Gamma$. By parts~\ref{EN:MU-SUM}
and~\ref{EN:MU-MULT} of Lemma~\ref{LEM:MU}
we have that $\mu(rr')\leq \mu(r)+\mu(r')$ and $\mu(a+a'+aa')\leq
\max(\mu(a),\mu(a'),\mu(a)+\mu(a'))=\mu(a)+\mu(a')$. It follows that
$\mu(rr')+ \Zed \mu(a+a'+aa') \leq \mu(r)+\Zed \mu(a)
+\mu(r')+\Zed\mu(a')\leq \Lambda+\Gamma$. Hence,
\[\frac{r}{1+a}\cdot \frac{r'}{1+a'}\>=\>\frac{rr'}{1+a+a'+aa'}\>\in\>
B[\Lambda+\Gamma]\;.\]
By additivity, it follows that property \ref{AX:CE} holds for~$\vcompbf$.

\par\smallskip
Let us show that property \ref{AX:CB} holds for~$\vcompbf$.
First, we prove that,
for every cut~$\Gamma$, $K = B[\Gamma] + \vring [\Gamma]$. We will
prove, by induction on the length
of~$c$, that for every $d \in R$, $0 \neq c \in R$, $n \in \Nat$ and
$\Gamma$ a cut of~$G$, $d/c^n$ splits at~$\Gamma$.

W.l.o.g., $v(c) = v(d) = 0$, and $c = 1 - a$, for some $a \in R$ with
$v(a) > 0$. If $a = 0$, then $d/c^n = d \in R$, and we are done.
Otherwise, let $\Theta := \mu(c) = \mu(a)$. Note that $\Theta > 0$.

Split $d = d_1 + d_2$ at $\Gamma$. Note that $v(d_2 / c^n) = v(d_2) >
\Gamma$. Thus, it suffices to split $d_1/c^n$ at~$\Gamma$, and
therefore, w.l.o.g., we can assume that $d = d_1$, and thus $\mu(d) \leq
\Gamma$. If $d = 0$ we are done, otherwise $\Gamma > 0$, and therefore
$\Gamma \geq \hat \Gamma$.

There are 4 cases: either $\Zed \Theta \leq \hat\Gamma$, or $\hat \Gamma
< \Zed \Theta < \Gamma$, or $\Zed \Theta = \Gamma$, or $\Zed \Theta >
\Gamma$.

If $\Zed \Theta \leq \hat\Gamma$, then
\[
\mu(d) + \Zed \mu(c^n) \leq \mu(d) + \Zed \mu(c) \leq \Gamma + \Zed
\Theta = \Gamma,
\]
and therefore $d/c^n \in A[\Gamma]$.

If $\Zed \Theta = \Gamma$, then $\Gamma = \hat \Gamma$, and we are in the
previous case.

If $\Zed \Theta > \Gamma$, we have $\Gamma < n_0 \theta_0$ for some $n_0
\in \Nat$ and $\theta_0 < \Theta$. Split $a = a_1 + a_2$ at $\theta^-$,
and define $c_1 := 1 - a_1$. Write
\[
d/c^n = d/(c_1 + a_2)^n = \frac{d/c_1^n}{(1 + a_2/c_1)^n} =
\sum_i m_{i, n} \frac{d a_2^i}{c_1^{i + n}},
\]
for some natural numbers $m_{i, n}$. Note that $\mathrm{lt}(c_1) <
\mathrm{lt}(c)$, and therefore, by induction on the length of~$c$, each
summand $x_i := m_{i,n} d a_2^i /c_1^{i + n}$ splits
at~$\Gamma$.
Moreover, for each $i \geq n_0$, $v(x_i) = v(m_{i, n}) +  i v(a_2) \geq n_0 \theta_0 >
\Gamma$, and therefore $x$ splits at~$\Gamma$.

If $\hat \Gamma < \Zed \Theta < \Gamma$, let $\Psi := \Gamma - \Zed
\Theta > 0$, and split $d = d_1 + d_2$ at~$\Psi$. Note that $\mu(d_1) +
\Zed\mu(c^n) \leq \Psi + \Zed\Theta \leq \Gamma$, and therefore $d_1/c^n
\in B[\Gamma]$. It remains to split $d_2 / c^n$. Let $\delta := v(d_2) >
\Psi$. By definition of $\Psi$, there exists $n_0 \in \Nat$, $\theta <
\Theta$, $\gamma > \Gamma$, such that $\delta \geq \gamma - n_0 \theta$.
Split $a = a_1 + a_2$ at $\theta^-$, and define $c_1 := 1 - a_1$.
As before,
\[
d_2/c^n = d_2/(c_1 - a_2)^n = \frac{d/c_1^n}{(1 - a_2/c_1)^n} =
\sum_i m_{i, n} \frac{d a_2^i}{c_1^{i + n}},
\]
and, by induction on the length of~$c$, each summand $x_i := m_{i,n} d_2
a_2^i /c_1^{i + n}$ splits at~$\Gamma$. Moreover, for every $i \geq
n_0$,
\[
v(x_i) = v(m_{i,n}) +  v(d_2) + i v(a_2) \geq \delta + n_0 \theta
\geq \gamma > \Gamma,
\]
and we are done.

\par\smallskip
In order to finish our proof, it now suffices to show that
$B[\Gamma]\cap \vring [\Gamma]=\{0\}$ for every cut~$\Gamma$. In the
next proposition, we will deduce a normal form for every non-zero
element $b\in B[\Gamma]$ that will prove that $v(b)\leq\Gamma$ so that
$b$ cannot lie in $\vring [\Gamma]$.
\end{proof}

\begin{proposition}
With the preceding definition of the tower $\vcompbf$, take $b\in
B[\Gamma]$ for some $\Gamma \in \vgh$. Then $b$ can be written in the
form
\[\sum_{i=1}^{k}\frac{r_i}{1+a_i}\]
with $a_i, r_i \in R$ such that $v(a_i) > 0$ and $\mu(r_i) + \Zed
\mu(a_i) \leq \Gamma$ for all $i$, and such that
\[v(r_1)<\ldots < v(r_k)
\;\mbox{ and }\;\Zed \mu(a_k)<\ldots < \Zed \mu(a_1)\;.\]
Moreover, $v(r_{i+1})>\mu(r_i)+\Zed\mu(a_i)$ for $1\leq i<k$.
\end{proposition}
\begin{proof}
As a first step, we prove the following. Suppose that
$a,a',r,r' \in R$ with $v(a) > 0$, $v(a') > 0$, $\mu(r) + \Zed \mu(a)
\leq \Gamma$ and $\mu(r') + \Zed \mu(a') \leq \Gamma$. If
\begin{equation}                            \label{mumu}
\max(\mu(r),\mu(r'))+\Zed\mu(a)+\Zed\mu(a')\leq\Gamma\;,
\end{equation}
then by parts~\ref{EN:MU-SUM} and~\ref{EN:MU-MULT} of Lemma~\ref{LEM:MU}
we have that
\begin{eqnarray*}
\lefteqn{\mu(r(1+a')+r'(1+a))+\Zed\mu(a+a'+aa')}\\
& \leq &
\max(\mu(r)+\mu(1+a'),\mu(r') +\mu(1+a))+\Zed (\mu(a)+\mu(a'))\\
& = & \max(\mu(r),\mu(r'))+ \Zed\mu(a)+\Zed\mu(a')\leq\Gamma
\end{eqnarray*}
since $\mu(1+a)=\mu(a)$ and $\mu(1+a')=\mu(a')$. This implies that
\[\frac{r}{1+a}\>+\>\frac{r'}{1+a'}\>=\>\frac{r(1+a')+r'(1+a)}{1+a+a'+aa'}
\>=\>\frac{r''}{1+a''}\]
with $a'', r'' \in R$, $v(a'') > 0$ and $\mu(r'') + \Zed \mu(a'') \leq
\Gamma$.

If (\ref{mumu}) does not hold, we must have that $\Zed\mu(a)\ne
\Zed\mu(a')$, and if $\Zed\mu(a)$ is the smaller of the two, we also
must have that $\mu(r)+\Zed\mu(a')>\Gamma$, which implies that
$\mu(r)>\mu(r')+\Zed\mu(a')=:\Theta$. In this case, split $r=s_1+s_2$ at
$\Theta$ so that $s_1\in A[\Theta]$ and $s_2\in \vring [\Theta]$. It
follows that $\mu(s_1)\leq\mu(r')+\Zed\mu(a')$ so that (\ref{mumu})
holds with $r$ replaced by $s_1$. So we can write
\[\frac{r}{1+a}\>+\>\frac{r'}{1+a'}\>=\>
\frac{s_1(1+a')+r'(1+a)}{1+a+a'+aa'}\>+\>\frac{s_2}{1+a}
\>=\>\frac{r''}{1+a''}\>+\>\frac{s_2}{1+a}\]
with $a'',a, r'',s_2 \in R$, $v(a'') > 0$, $v(a) > 0$, $\mu(r'') + \Zed
\mu(a'') \leq \Gamma$ and $\mu(s_2) + \Zed \mu(a) \leq \Gamma$. We have
that
\[\mu(r'')=\mu(s_1(1+a')+r'(1+a))\leq\max(\mu(s_1)+\mu(1+a'),
\mu(r')+\mu(1+a))\leq\Theta,\]
hence if $r''\ne 0$, then $v(s_2)> v(r'')$.

Let us also show that
\[\Zed\mu(a)<\Zed\mu(a'')\leq\Zed\mu(a')\;.\]
Split $a'=a'_1+a'_2$ at $\Zed\mu(a)$ so that $a'_1\in A[\Zed\mu(a)]$ and
$a'_2\in\vring [\Zed\mu(a)]$; since $\Zed\mu(a)<\Zed\mu(a')$, we must
have that $a'_2\ne 0$. Then
\[1+a''=(1+a)(1+a')=(1+a)(1+a'_1+a'_2)=(1+a)(1+a'_1)
+a'_2+aa'_2.\]
Since $(1+a)(1+a'_1)\in A[\Zed\mu(a)]$ and $0\ne
a'_2+aa'_2\in \vring [\Zed\mu(a)]$, we find that $\Zed\mu(a)<\mu(a'')$
and hence $\Zed\mu(a)<\Zed\mu(a'')$. On the other hand, $\Zed\mu(a'')=
\Zed \mu(a+a'+aa') \leq\Zed\max(\mu(a),\mu(a'),\mu(a)+\mu(a'))=
\Zed(\mu(a)+\mu(a'))=\Zed\mu(a')$.

\par\smallskip
Every non-zero element $b\in B[\Gamma]$ can be written as
\[b\>=\>\sum_{i=1}^{k}\frac{\tilde{r}_i}{1+\tilde{a}_i}\]
with $\tilde{a}_i, \tilde{r}_i \in R$ such that $v(\tilde{a}_i) > 0$ and
$\mu(\tilde{r}_i) + \Zed \mu(\tilde{a}_i) \leq \Gamma$ for all $i$.
Suppose that this is a representation of $b$ with minimal $k$. Then it
follows that for any choice of $i,j$ such that $1\leq i<j\leq k$,
(\ref{mumu}) cannot hold for $r_i,r_j,a_i,a_j$ in the place of
$r,r',a,a'$, respectively. So we know that all $\Zed\mu(\tilde{a_i})$
are distinct, and we may w.l.o.g.\ assume that $\Zed \mu(\tilde{a}_k)
<\ldots < \Zed \mu(\tilde{a}_1)$.

Suppose that $k\geq 2$. Having that $\Zed \mu(\tilde{a}_k)<\Zed
\mu(\tilde{a}_{k-1})$, we apply the above procedure to
$\tilde{r}_k,\tilde{r}_{k-1}$, $\tilde{a}_k, \tilde{a}_{k-1}$ in the
place of $r,r',a,a'$, respectively. The elements $r''$ and $s_2$
obtained here cannot be zero since otherwise, $k$ wouldn't have been
minimal. We set $r_k=s_2$ and replace $\tilde{r}_{k-1}$ by $r''$ and
$\tilde{a}_{k-1}$ by $a''$. By construction, $\Zed \mu(a_k)<\Zed
\mu(\tilde{a}_{k-1})$ and $v(r_k)>\mu(\tilde{r}_{k-1})+ \Zed
\mu(\tilde{a}_{k-1})$. And we still have that $\Zed \mu(\tilde{a}_{k-1})
<\Zed \mu(\tilde{a}_{k-2})$. So now we repeat the above procedure with
$\tilde{r}_{k-1},\tilde{r}_{k-2}$, $\tilde{a}_{k-1}, \tilde{a}_{k-2}$ in
the place of $r,r',a,a'$, respectively. We note that the non-zero
element $s_2$ we obtain this time, which will become our $r_{k-1}$,
satisfies $\mu(s_2)=\mu(\tilde{r}_{k-1})\leq \mu(\tilde{r}_{k-1})+ \Zed
\mu(\tilde{a}_{k-1}) <v(r_k)$. This yields $v(r_{k-1})<v(r_k)$, and from
now on we can proceed by descending induction. The element $r''$ found
in the last step will then be our $r_1\,$.
\end{proof}

\subsection{Extension to immediate field extensions} \label{5.3}
Let $\vcompf$ be a (weak) tower of complements for $\KF$, and $\FF$ an
\emph{immediate} extension of $\KF$. The aim of this subsection is to
extend $\vcompf$ to a (weak) tower of complements for $\FF$, under some
condition on the extension $\extension{\FF}{\KF}$. We need to study
further properties of the map $\mu$ (cf.\ Definition \ref{MU}).
\begin{lemma}\label{LEM:MU-INVERSE}
Let $0 \neq a \in \KF$ such that $\val a = 0$, $b := \frac{1}{a}$,
$\Lambda := \Zed\mu(a)$.
Then, $\mu(b) \leq \Lambda$.\footnote{That is, $b \in \vcompL$.}
\end{lemma}
\begin{proof}
If $\Lambda = +\infty$, there is nothing to prove.
Otherwise, decompose $b = b_1 + b_2$, with $b_1 \in \vcompL$,
$\val(b_2) > \Lambda$.
Hence, $\mu(a b_1) \leq \Lambda + \Lambda = \Lambda$,
while $\val(a b_2) = \val(b_2) > \Lambda$.
Moreover, $1 = a b = a b_1 + a b_2$, hence, by the uniqueness of the
decomposition of $1$ at $\Lambda$, we get $a b_2 =0$, that is,
$b = b_1 \in \vcompL$.
\end{proof}
\begin{lemma}
Let $a \in \KF$ be of the form $a = a' + c t^\lambda$,
where $c \in \ks$, and $0 = \val a' < \mu a' < \lambda$.
Let $\Lambda := \Zed\lambda^+$, $b := \frac{1}{a}$.
Then, $\mu(b) = \Lambda$.
\end{lemma}
\begin{proof}
Note that $\mu a = \lambda^+$.
By the previous lemma, $\mu b \leq \Lambda$.
Suppose, for a contradiction, that $\Gamma :=\mu b < \Lambda$.
Choose $n \in \Nats$ such that $(n - 1) \lambda > \Gamma$.
Define
\[
b' := \frac{1- (1 - \frac{a}{a'})^n}{a} = \sum_{i = 0}^{n-1}
\tbinom{n}{i} a'^i (-a)^{n - 1 - i}.
\]
Since $\mu a' < \lambda$ and $\mu a = \lambda^+$,  $\mu b' \leq (n - 1)
\lambda^+$.
Thus,
\[
\mu (b - b') \leq \max \set{\mu b, \mu b' } \leq \max \set{\Gamma,
(n - 1) \lambda^+} = (n - 1)\lambda^+.
\]
Therefore, $\mu\bigl(a (b - b') \bigr) \leq \mu a + \mu(b-b') \leq
n\lambda^+$.

Moreover,
\[
a (b - b') = \bigl( \frac{a' - a}{a'} \bigr)^n =
\bigr( \frac{-c t^\lambda}{a'} \bigr)^n.
\]
Hence, $\val(b - b') = \val\Pa{a (b - b')} = n \lambda$.
Therefore, $b' = b + (b' - b)$, with $b \in \vcompG$ and $b' - b \in
\vring[n \lp] \subseteq \vringG$. Thus, since $b' \in \vcomp[(n-1)\lp]$,
by~\ref{EN:COMPL-BASIC-DEC}, $b' - b \in \vcomp[(n-1) \lp] \subseteq
\vcomp[n \lm]$. Finally, $b - b'\in \vring[n \lm] \cap \vcomp[n \lm] =
(0)$, so $b = b'$, that is, $a = a'$, a contradiction.
\end{proof}

\begin{definizione}
For $a \in \FF$, we define
\[
\La := \set{\val(a-c): c\in\KF}^+ \in \vgh.
\]

Note that if $a \in \Kf$ then $\La = + \infty$.
\end{definizione}
\begin{lemma}\label{LEM:MU-LAMBDA}
Let $\bR$ be a sub-ring of $\,\FF$ such that $\Kf
\subseteq \bR \subseteq \FF$, and $a \in \bR \setminus \Kf$.
Let $\vcompbf$ be a \wtoc for $\bR$ extending $\vcompf$.
Then, $\La \leq \mu(a)$.
\end{lemma}
\begin{proof}
Suppose not.
Let $\Gamma := \mu(a)$ and $c \in \Kf$ such that
\[
\Gamma < v(a - c). 
\]
Decompose $c$ at $\Gamma$: $c = c_1 + c_2$, with $c_1 \in \vcompG$
and $c_2 \in \vringG$.
Since $a - c_1 = (a - c) + c_2$, we have
\[
\val(a - c_1) \geq \min \Pa{ v(a - c), v(c_2)} > \Gamma.
\]
Moreover, $a \in \vcompbG$ by definition of $\Gamma$, and $c_1 \in
\vcompbG$, thus $a - c_1 \in \vcompbG \cap \vringG$.
Therefore, $a - c_1 = 0$, hence $a \in \Kf$, which is absurd.
\end{proof}
\begin{corollary}\label{COR:COMPLETION}
If $\FF$ is contained in the completion of $\KF$, then $\vcompf$ itself
is the unique \wtoc which $\FF$ extends $\vcompf$ from $\KF$ to $\FF$.
Moreover, $\vcompbf$~is multiplicative if and only if $\vcompf$ is.
\end{corollary}
\begin{proof}
\textbf{\textsf{Existence}}.
Define $\vcompbG = \vcompG$.
All the properties of (weak) \towcomp for $\vcompbf$ are obvious,
except possibly the fact that $\vcompbG + \vringG = \FF$.
Let $a \in \FF$, and $c \in \Kf$ such that $\val(a - c) > \Gamma$.
Decompose $c = c_1 + c_2$, with $c_1 \in \vcompG$ and $\val(c_2) > \Gamma$.
Then, $a = c_1 + (a - c + c_2)$.
Therefore, $c_1 \in \vcompbG$ and $a - c + c_2 \in \vringG$.

\textbf{\textsf{Uniqueness}}.
If for some $\Gamma$ we would have that there is some $a\in \vcompbG
\setminus \vcompG$ then we could decompose $a = a_1 + a_2$ with $a_1
\in \vcompG$ and $a_2\in\vringG$. But then $0\ne a-a_1\in \vcompbG
\cap\vringG$, contradiction.
\end{proof}
\sn
We will consider the case when $\FF := \KF(a)$ for some
$a\in\FF\setminus\KF$. Let $\isequence$ be a pseudo Cauchy sequence in
$\KF$, without a limit in $\KF$, and converging to $a\in\FF$. Note that
in this case
\[
\La = \set {\val(a - \anu): \nu\in I}^+.
\]
\mn
We will say that a certain property of the sequence members $\anu$ holds
\emph{eventu\-al\-ly} (or \emph{for $\nu$ large enough}) if there is
some $\nu_0\in I$ such that it holds for all $\nu\geq\nu_0$, $\nu\in I$,
and we will say that it holds \emph{frequently} if for all $\nu'\in I$
it holds for some $\nu\in I$ with $\nu\geq\nu'$.
\mn
We will assume the reader to be familiar with the basic theory of
pseudo Cauchy sequence as outlined in \cite{Kap}.
If $\isequence$ is of algebraic type, let $m$ be the degree of a minimal
polynomial for it. Otherwise, let $m:=+\infty$. In both cases, $m$ is
the maximum such that any polynomial $p(X)\in\KF[X]$ of degree less than
$m$ will satisfy $\val\Pa{p(\anu)} = \val\Pa{p(a_\mu)}$ for $\nu, \mu$
large enough.
\sn
\begin{fhypothesis}                         \label{FH}
\begin{enumerate}
\renewcommand{\theenumi}{\alph{enumi})}
\renewcommand{\labelenumi}{\alph{enumi})}
\item Either $\isequence$ is of transcendental type
(and therefore $a$ is transcendental over $\KF$),
\item or $a$ is a root of a minimal polynomial for $\isequence$.
\end{enumerate}
\end{fhypothesis}
In either case, $[\FF:\KF] = m$.
Let $\LF:=\KF[a]$, the ring generated by $\KF$ and $a$. Every element of
$\LF$ can be written in a unique way as a polynomial in $a$ of degree
less than $m$. Moreover, if $a$ is algebraic (case b)), then $\LF=\FF$.
\sn
Decompose each $\anu = \anu' + \anu''$, where $\anu'\in\vcomp[\La]$,
$\anu''\in\vring[\La]$.
\sn
\begin{remark}
$\Pa{\anu'}_{\nu\in I}$ is a pseudo Cauchy sequence with the same
limits as~$\isequence$.
\end{remark}
Hence, we can use $\Pa{\anu'}_{\nu\in I}$ instead of $\isequence$,
and, \wloG, we can assume that $\anu \in \vcomp[\La]$.
\sn
\begin{definizione}\label{DEF:EXT-RING}
If $\La < +\infty$, for any $\Gamma \in \vgh$, define $\vcompbG$ to be
the $\kf$-linear subspace of~$\LF$ generated by $\vcomp[\Lambda]$,
together with the monomials of the form
\begin{equation}\label{EQ:EXT-RING}
c a^n, \quad \text{where } n < m, \text{ and } c \in \vcomp[\Gamma - n\La].
\end{equation}
If instead $\La = + \infty$, define $\vcompbG := \vcompG$.
\end{definizione}
\begin{remark}\label{REM:EXT-RING-UNIQUE}
Let $n < m$ and $c \in \Kf$.
Then, $c a^n \in \vcompbG$ if and only if $c \in \vcomp[\Gamma - n \La]$.
\end{remark}
\begin{proof}
The ``if'' direction follows from the definition of $\vcompbf$.
Conversely, assume that $c a^n = c_1 a_1^{n_1} + \dotsb + c_l
a_l^{n_l}$, with $c_i \in \vcomp[\Gamma - n_i \La]$. Up to permuting and
adding together some of the $c_i$'s, we can assume that $n_1 < n_2 <
\dots < n_l < m$. Since the degree of
$a$ over $\KF$ is~$n$, if $c \neq 0$, then $\exists ! k \leq l$ such
that~$n = n_k$; moreover, $c_i = 0$ for every $i \neq k$, and $c =
c_{n_k}$. The conclusion follows.
%
\end{proof}
\begin{lemma}
Assume that the family $\vcompf$ is a tower of complements for $\KF$.
Let $b \in \vcomp[\Gamma]$, $c \in \vcomp[\Lambda - n\Gamma]$.
Then, $c b^n \in \vcompL$.
\end{lemma}
\begin{proof}
Let $\Theta := (\Lambda - n\Gamma) + n\Gamma$.
By Axiom~\ref{AX:CE}, $cb^n \in \vcomp[\Theta]$.
By Proposition~\ref{PROP:N-GAMMA}, $\Theta \leq \Lambda$, therefore,
by Axiom~\ref{AX:CD}, $cb^n \in \vcompL$.
\end{proof}
\begin{proposition}\label{PROP:RING-EXT}
Assume that the family $\vcompf$ is a tower of complements for~$\KF$.
Then the family $\vcompbf := \set{\vcompbG: \Gamma \in \vgh}$ defined
above is a weak tower of complements for~$\LF$ extending $\vcompf$.
\end{proposition}
\begin{proof}
When $\La = + \infty$, our assertion follows from
Corollary~\ref{COR:COMPLETION}; therefore, we assume that $\La < +
\infty$. Taking $n = 0$ in \eqref{EQ:EXT-RING}, we see that $\vcompG
\subseteq \vcompbG$, hence $\vcompbf$ extends $\vcompf$.
\sn
Axiom~\ref{AX:CA} is trivial, and Axiom~\ref{AX:CC} is a consequence of
the fact that $\vcompbf$ extends~$\vcompf$.
Since $\Gamma - n\La$ is an increasing function of~$\Gamma$,
Axiom~\ref{AX:CD} is also trivial.
 \sn
Axiom~\ref{AX:CB} splits into two parts: $\vcompbG + \vringG = \LF$ and
$\vcompbG \cap \vringG = (0)$.
The first part is equivalent to that every polynomial $p(a) \in\KF[a]$
of degree $n<m$ can be decomposed as $p(a) = p_1(a) + p_2(a)$, with
$p_1(a) \in \vcompbG$, and $p_2(a) \in \vringL$. We will prove this by
induction on~$n$. \Wlog, we can assume that $p(a)$ is a monomial~$ca^n$.
\sn
If $\La < +\infty$, decompose $c = c_1 + c_2$, $c_1 \in \vcomp[\Gamma -
n\La]$, $c_2\in\vring[\Gamma - n\La]$. By definition, $c_1 a^n \in
\vcompbG$. Moreover, $\val(c_2) \geq \gamma - n\lambda$, for some
$\gamma > \Gamma$, $\lambda < \La$. Hence, $\val(c_2) \geq \gamma
-\val\Pa{(a - \anu)^n}$ for some $\nu \in I$. Therefore, $c_2(a- \anu)^n
\in \vringG$. Finally, we get
\[
c a^n = \underbrace{c_1 a^n}_{\in\vcompbG} +
\underbrace{c_2 (a - \anu)^n}_{\in\vringG} + c_2 \Pa{a^n - (a - \anu)^n}.
\]
The polynomial (in $a$) $c_2\Pa{a^n - (a-a_\nu)^n}$ has degree less
than~$n$. Hence, by inductive hypothesis, it can be written as $b_1 +
b_2$, with $b_1 \in \vcompbG$, $b_2 \in \vringG$, and we get the
decomposition of $c a^n$.
\sn
For the second part, assume that
%
%
$q(a) := \sum_{n<m}c_n a^n \in \vringG \cap \vcompbG$, that is, for
every $n$, $c_n \in \vcomp[\Gamma - n\La]$, and $\val\Pa{q(a)}>\Gamma$.
Choose $\nu\in I$ large enough such that $\val\Pa{q(a)} =
\val\Pa{q(\amu)}$.
Since $\amu \in \vcomp[\La]$, $q(\amu) \in \vcompG$, because, by the
above lemma, each summand $c_n \amu^{n}$ is in $\vcompG$. Hence,
$q(\amu) = 0$ for every $\mu$ large enough. Therefore, $q(X)$ has
infinitely many zeroes, so $q(X) = 0$, and in particular $q(a) = 0$.
\sn
Finally, to prove Axiom~\ref{AX:CF}, observe that
\begin{eqnarray*}
c a^n \in \vcompbG & \Longleftrightarrow & c\in\vcomp[\Gamma - n\La]
\>\Longleftrightarrow\> t^\gamma c \in \vcomp[\gamma + \Gamma - n\La]\\
 & \Longleftrightarrow & t^\gamma c a^n \in \vcompb[\gamma + \Gamma].
\hfill
\end{eqnarray*} \qedhere
\end{proof}
\begin{lemma}\label{LEM:MU-POLY}
If $n < m$ and $c \in \KF$, then $\mu (c a^n) = \mu(c) + n \La$.
More generally, for every $c_0, \dotsc, c_{m-1} \in \KF$,
\[
\mu( \sum_{0 \leq n < m} c_n a^n) = \max_{n < m} \set{ \mu(c_n) + n \La}.
\]
In particular, $\mu(a) = \La$.
\end{lemma}
\begin{proof}
By Lemma~\ref{LEM:MU} and Lemma~\ref{LEM:SUM-DIFF},
\[\begin{aligned}
\mu(c a^n) &= \inf \set{ \Lambda: c a^n \in \vcompbL} =
\inf \set{ \Lambda: c \in \vcomp[\Lambda - n \La] } =\\
& =  \inf \set{ \Lambda: \mu(c) \leq \Lambda - n \La} =
\inf \set{ \Lambda: \mu(c) + n \La \leq \Lambda}\\
& = \mu(c) + n \La.
\end{aligned}\]
The second point is now obvious.
\end{proof}
For the rest of this section, we will assume that $\vcompf$ is a T.o.C.,
and that $\vcompbf$ is built as in Proposition~\ref{PROP:RING-EXT}.
unless explicitly said otherwise.
\begin{lemma}\label{LEM:EXT-UNIQUE}
The family $\vcompbf$ is the unique \wtoc on $\LF$ such that:
\begin{enumerate}
\item $\mu(a) \leq \La$;
\item for every $n < m$ and $c \in \Kf$, $\mu(c a^n) \leq\mu(c)+n\mu(a)$.
\end{enumerate}
\end{lemma}
\begin{proof}
Let $\vcompbf'$ be another $\wtoc$ for $\LF$ satisfying the conditions.
By\linebreak
Lemma~\ref{LEM:MU-LAMBDA}, $\mu(a) = \La$.
Moreover, if $n < m$ and $c \in \vcomp[\Gamma - n\La]$,
\[
c a^n \in \vcompbp{(\Gamma - n \La) + n \La} \subseteq \vcompbp{\Gamma}.
\]
Therefore, $\vcompbG \subseteq \vcompbp{\Gamma}$, and thus $\vcompbf' =
\vcompbf$.
\end{proof}
\begin{lemma}\label{LEM:EXT-MU}
Let $q(X) \in \Kf[X]$ such that $\deg q = n < m$.
Then,
\begin{enumerate}
\item
If $q(X) \neq 0$, then $\val\Pa{q(a) - q(\anu)} > \val(q(a)) =
\val(q(\anu))$ eventually.
\item\label{EN:EXT-MU-1} $\mu(q(\anu)) \leq \mu(q(a))$ for every
$\nu \in I$.
\item\label{EN:EXT-MU-2} $q(a) \in \vcompbG$ if and only if $q(\anu)
\in \vcompG$ eventually.
\item\label{EN:EXT-MU-3} $\mu(q(a)) = \lim_{\nu \in I} \mu(q(\anu))$.
\item\label{EN:DECOMP} If $q(X) = q_1(X) + q_2(X)$, with $q_1(a) \in
\vcompG$ and $q_2(a) \in \vringG$ such that $\deg q_i<m$, then
$\deg q_i \leq n$.
\end{enumerate}
\end{lemma}
\begin{proof}
For the first assertion, see \cite{Kap}.

\par\smallskip
By Lemma~\ref{LEM:MU-POLY},
\[
\mu(q(a)) = \max_{i\leq n} \set{\mu(b_i) + i \La}.
\]
Since $\vcompf$ is multiplicative, $\mu(q(\anu)) \leq
\max_i\set{\mu(b_i) + i \mu(a)}$. Therefore, since $\mu(a_\nu) \leq
\La$, we have $\mu(q(\anu)) \leq \mu(q(a))$.

\par\smallskip
Let $q(\anu) \in \vcompG$ eventually. Decompose $q(a) = q_1(a) +
q_2(a)$, with $q_1(a) \in \vcompbG$, and $q_2(a) \in \vringG$. Then,
$q_1(\anu)$ and $q(\anu)$ are in $\vcompG$ eventually, and therefore
$q_2(\anu) = 0$ eventually. Thus, $q_2(X)$ has infinitely many zeroes,
hence $q_2 = 0$. Therefore, we have proved assertion~\ref{EN:EXT-MU-2}.

\par\smallskip
For assertion~\ref{EN:EXT-MU-3}, let $\Gamma':= \liminf_{\nu \in I}
\mu(q(\anu))$. Then $\Gamma' \leq \mu(q(a))$. If, for a contradiction,
$\Gamma' < \mu(q(a))$, choose $\Gamma >\Gamma'$ such that $\mu(q(\anu))
\leq \Gamma$ frequently. Decompose $q(a) = q_1(a) + q_2(a)$, with
$q_1(a) \in \vcompG$ and $0 \neq q_2(a) \in \vringG$. Then by the third
assertion, $q_1(\anu) \in \vcompG$ and $q_2(a_\nu) \in \vringG$
eventually, hence $q_2(a_\nu) = 0$ frequently, and therefore $q_2 = 0$,
which is absurd.

\par\smallskip
Now consider $\Gamma'':= \limsup_{\nu \in I} \mu(q(\anu))$. Then
$\Gamma''\geq \liminf_{\nu \in I} \mu(q(\anu))=\mu(q(a))$. By our second
assertion, $\Gamma''\leq\Gamma$; thus, $\lim_{\nu \in I} \mu(q(\anu))$
always exists and is equal to $\mu(q(a))$.

\par\smallskip
Assertion~\ref{EN:DECOMP} is a consequence of the proof of
Proposition~\ref{PROP:RING-EXT}.
\end{proof}

\begin{lemma}\label{LEM:BASIC-MONOMIAL}
Let $\La < +\infty$, $n < m$, $d \in \vcomp[\La]$.
\begin{enumerate}
\item If $c\in\vcomp[\Lambda - n \La]$, then $c (a + d)^n \in \vcompbL$.
\item If $c\in\vcompL$, then $c (a + d)^n \in \vcompb[\Lambda + n \La]$.%
\footnote{It is enough that $c \in \vcomp[(\Lambda + n\La) -n\La]$.}
\end{enumerate}
In particular, if $0 <n <m$, then $(a + d)^n \in \vcompb[n \La]$.
\end{lemma}
\begin{proof}
For the first part, write
\begin{equation}\label{EQ:MONOMIALS}
c (a + d)^n = \sum_{i<n} \tbinom{i}{n} c d^{n-i} a^i.
\end{equation}
Note that $\tbinom{i}{n} \in \kf$.
Hence, by Proposition~\ref{PROP:N-GAMMA},
\[
\tbinom{i}{n} c d^{n-i} \in \vcomp[\Lambda - n \La] \vcomp[(n-i)\La]
\subseteq \vcomp[(\Lambda - n \La) + (n-i) \La] \subseteq
\vcomp[\Lambda - i \La].
\]
Therefore, by definition, each of the summands of~\eqref{EQ:MONOMIALS}
is in $\vcompbL$.
\sn
For the second part, define $\Lambda' := \Lambda + n\La$. By the first
part, to prove the conclusion, it suffices to show that $c \in
\vcomp[\Lambda' - n \La]$. By Proposition~\ref{PROP:N-GAMMA}, $\Lambda'
- n\La \geq \Lambda$, and the conclusion follows.
\end{proof}
\begin{corollary}\label{COR:MU-MONOMIAL-2}
If $n < m$ and $c \in \KF$, then, for every $\nu \in I$,
$\mu \Pa{c (a - \anu)^n} = \mu(c) + n \La = \mu(c a^n)$.
\end{corollary}
\begin{proof}
Since $\anu \in \vcomp[\La]$, we have that
\[
c(a - \anu)^n \in \vcompb[\mu(c) + \ n \La].
\qedhere
\]
\end{proof}
\begin{proposition}\label{PROP:RING-MULT}
Let $p(X), q(X) \in\KF[X]$. Assume that $\deg p + \deg q < m$, and that
$p(a) \in \vcompbG$, $q(a) \in \vcompbL$. Then, $p(a)q(a) \in
\vcompb[\Gamma +\Lambda]$.
\end{proposition}
\begin{proof}
This is trivial if $\La = + \infty$. Otherwise, \wloG $p(a)$ and $q(a)$
are monomials of the form $c a^n$ and $d a^r$ respectively, with $c \in
\vcomp[\Gamma - n\La]$, $d\in\vcomp[\Lambda - r\La]$, and $n + r < m$.
Then, $p(a)q(a) = c d a^{n+r}$. Moreover, by
Corollary~\ref{COR:N-GAMMA},
\[
c d \in \vcomp[\Gamma - n\La] \vcomp[\Lambda - r\La] \subseteq
\vcomp[(\Gamma - n\La)+(\Lambda - r\La)] \subseteq
\vcomp[(\Gamma + \Lambda) - (n + r)\La],
\]
and the conclusion follows.
\end{proof}
\begin{corollary}\label{COR:FIELD-MULT}
If, in the above proposition, $m=\infty$ (case a)),
then $\vcompbf$ is a tower of complements
for $\LF$ (that is, it satisfies Axiom~\ref{AX:CE}).
\end{corollary}

\par\smallskip
Now our aim is to extend a tower of complements from $\LF = \KF[a]$ to
$\FF = \KF(a)$. We could just use Proposition~\ref{quot}, but we we want
to show how the construction works in the present special situation,
where $\vcompf$ is a tower of complements for~$\KF$, $a$ an element
of transcendental type over $\KF$, satisfying case a) of the Fundamental
Hypothesis (\ie, $m = \infty$). Let $\vcompbf$ be the \wtoc defined
in~\ref{DEF:EXT-RING} (it is a \towcomp by Prop.~\ref{PROP:RING-EXT} and
Corollary~\ref{COR:FIELD-MULT}). Lemma~\ref{LEM:EXT-MU} shows that an
equivalent definition of $\vcompbf$ is given by:\\
for every $q(X) \in \KF[X]$,
$q(a) \in \vcompbG$ if and only if $q(\anu) \in \vcompG$ eventually.
\mn
The above definition makes sense also in~$\FF$. Therefore, we define
\par\smallskip\noindent
$\vcompcG:= \set{r(a): r(X) \in \KF(X) \et r(\anu) \in \vcompG
\text{ eventually}}$, and
\par\noindent
$\vcompcf := \Pa{\vcompcG}_{\Gamma\in \vgh}$.
\begin{lemma}\label{LEM:TOWER-TRASCENDENTAL}
The above defined family $\vcompcf$ is a \towcomp for~$\FF = \KF(a)$.
\end{lemma}
\begin{proof}
All axioms are immediate to show, except~\ref{AX:CB} and~\ref{AX:CE}.
The important fact used in the proof is that, for every $r(X) \in \KF(X)$,
$\val\Pa{r(a)} = \val\Pa{r(\anu)}$ eventually.
\sn
Let us prove Axiom~\ref{AX:CE} first. Let $c \in C[\Gamma]$ and $c' \in
C[\Gamma']$. Write $c = r(a)$ and $c' = r'(a)$, for some $r, r' \in
\KF(X)$. By definition, $r(\anu) \in A[\Gamma]$ and $r'(\anu) \in
A[\Gamma']$ eventually. Thus, $rr'(\anu) \in A[\Gamma + \Gamma']$
eventually, and therefore $c c' \in C[\Gamma + \Gamma']$.
 \sn
For Axiom~\ref{AX:CB}, we prove first that $C[\Gamma] \cap
\vring[\Gamma] = \{ 0 \}$. Let $c = r(a) \in C[\Gamma] \cap
\vring[\Gamma]$. Hence, $r(\anu) \in A[\Gamma]$ eventually.
Moreover, $\val\Pa{r(a)} = \val\Pa{r(\anu)}$ eventually, and therefore
$r(\anu) \in \vring[\Gamma]$ eventually. Thus, $r(\anu) = 0$ eventually,
hence $r(X)$ has infinitely many zeros, and so $r = 0$.
\sn
Now we prove that $\FF = \vcompG + \vringG$ for every $\Gamma \in \vgh$.
We have to show that every element $d/c\in\FF$, where $d \in \LF$ and $0
\neq c \in \LF$, can be split at any given $\Gamma$. But if $\gamma=
v(d/c)$, it suffices to show that $t^{-\gamma}d/c$ can be split at
$-\gamma+\Gamma$. Hence we may assume w.l.o.g.\ that $v(d)=v(c)=0$.
We will prove, by induction on the length
of~$c$ that, for every $0 < k \in \Nat$, and for every $\Gamma \in
\vgh$, $b := d/c^k$ splits at~$\Gamma$.
\sn
Let $d = p(a)$ and $c = q(a)$, for some $p, q \in \KF[X]$. By part 2
of Lemma~\ref{LEM:EXT-MU}, we then have that $\mu(p(\anu))\leq\mu(d)$
and $\mu(q(\anu))\leq\mu(c)$ for all $\nu\in I$. By part 4 of the same
lemma, $\mu(c)= \lim_{\nu\in I}\mu(q(\anu))$. Since $v(c) = 0$, $\mu(c)$ is
a positive cut. Therefore, $\mu(q(\anu))$ is a positive cut $\leq\mu(c)$
eventually, showing that $\Zed \mu(q^k(\anu))\leq\Zed \mu(c)$
eventually. Also, $v(p(\anu))=v(d)=0$ eventually.
\sn
\begin{claim}\label{CL:MU-FRACTION}
If $\mu(d)+ \Zed \mu(c) \leq \Gamma$, then $d/c^k \in \vcompcG$.
\end{claim}
\sn
Indeed, by the facts shown above, $\mu\Pa{p(\anu)} +\mu\Pa{q^k(\anu)}
\leq \mu(d)+\Zed \mu(c) \leq \Gamma$ eventually. Thus, by
Lemma~\ref{LEM:MU-INVERSE}, $\mu\Pa{(p/q^k)(\anu)} \leq \Gamma$
eventually. Therefore, $b \in \vcompcG$ by definition, and we are done.
\sn
Now split $d =: d_1 + d_2$ at~$\Gamma$. Note that $v(d_2/c^k) = v(d_2) >
\Gamma$, hence $d_2/c^k \in \vringG$. Therefore, it suffices to prove
that $d_1/c^k$ splits at~$\Gamma$; thus, we can assume that $d = d_1$,
that is, $\mu(d) \leq \Gamma$. In view of this, Claim 1 constitutes the
induction start for our induction on the length of $c$.
\sn
We write $c = 1 - \varepsilon$ for some $\varepsilon \in \LF$ with
$\val(\varepsilon) > 0$. We assume that $\varepsilon \ne 0$ because
otherwise, $b \in \LF$, and there is nothing to show. Define $\Theta :=
\mu(c) = \mu(\varepsilon)$. Note that $\Theta > 0$ and $q(\anu) \in
\vcomp[\Theta]$ eventually. There are 4 cases: $\Zed \Theta \leq
\ig\Gamma$, or $\ig \Gamma < \Zed\Theta < \Gamma$, or $\Zed\Theta =
\Gamma$, or $\Zed \Theta > \Gamma$.
\sn
If $\Zed \Theta \leq \ig \Gamma$, then, by Claim~\ref{CL:MU-FRACTION},
$b \in \vcompcG$, and we are done.
\sn
If $\Zed \Theta = \Gamma$, then $\Zed \Theta = \ig\Gamma$, and we are
in the previous case.
\sn
If $\Zed\Theta > \Gamma$ we have that $\Gamma < n_0 \theta_0$ for some
$\theta_0 < \Theta$ and $n_0\in \Nat$. Split $\varepsilon =:
\varepsilon_1 + \varepsilon_2$ at $\theta_0^-$: \ie, $\mu(\varepsilon_1)
< \theta_0 \leq \val(\varepsilon_2)$. Define $c_1 := 1 - \varepsilon_1$.
Note that, since $\theta_0 < \mu(c)$, $\lt(c_1) < \lt(c)$. Moreover,
\[
b = \frac d {c^k} = \frac d {(c_1 - \varepsilon_2)^k} =
\frac{d/c_1^k}{\Pa{1 - (\varepsilon_2/c_1)}^k} =
\sum_{i \in \Nat} n_{i, k} \frac{d \varepsilon_2^i} { c_1^{i + k}},
\]
for some natural numbers $n_{i,k}$.
Since $\lt(c_1) < \lt(c)$, then, by induction on $\lt(c)$, each summand
$n_{i, k} d \varepsilon_2^i / c_1^{i + k}$ splits at~$\Gamma$.
Moreover,
\[
\val\Pa{\sum_{i \geq n_0}  n_{i, k} \frac{d \varepsilon_2^i} {c_1^{i + k}}}
\geq \val\Pa{\frac{d \varepsilon_2^{n_0}} {c_1^{n_0 + k}}}
=  \val(\varepsilon_2^{n_0}) \geq n_0 \theta_0 > \Gamma,
\]
and therefore $b$ splits at~$\Gamma$.
\sn
Finally, if $\ig \Gamma < \Zed \Theta < \Gamma$, let $\Psi := \Gamma -
\Zed \Theta > 0$, and split $d = d_1 + d_2$ at~$\Psi$. Note that
$\mu(d_1) + \Zed\mu(c) \leq \Psi + \Zed\Theta \leq \Gamma$, and
therefore $d_1/c^k \in \vcompcG$. It remains to split $d_2 / c^k$. Let
$\delta := v(d_2) > \Psi$. By definition of $\Psi$, there exists $n_0
\in \Nat$, $\theta_0 < \Theta$, $\gamma > \Gamma$, such that $\delta
\geq \gamma - n_0 \theta_0$. Split $\varepsilon = \varepsilon_1 +
\varepsilon_2$ at $\theta^-$, and define $c_1 := 1 - \varepsilon_1$. As
before,
\[
d_2/c^k = d_2/(c_1 - \varepsilon_2)^n = \frac{d_2/c_1^k}{(1 -
\varepsilon_2/c_1)^k} =
\sum_i n_{i, k} \frac {d_2 \varepsilon_2^i} {c_1^{i + k}}\>,
\]
and by induction on the length of~$c$, each summand $n_{i,k}
d_2 \varepsilon_2^i /c_1^{i + k}$ splits at~$\Gamma$.
Moreover, for every $i \geq n_0$,
\[
v(n_{i,k} d_2 \varepsilon_2^i /c_1^{i + k} ) = v(n_{i,k}) + v(d_2) + i
v(\varepsilon_2) \geq \delta + n_0 \theta \geq \gamma > \Gamma,
\]
and we are done.
\end{proof}
\begin{proposition}\label{PROP:ALG-MULT-NEC}
Assume that $m < \infty$ (case b)). Let $p(X) := \sum_{n = 0}^m b_n
X^n$ $\in \KF[X]$ be the minimal polynomial of $a$ over~$\KF$ (thus,
$b_m=1$). Then a necessary and sufficient condition for the family
$\vcompbf$ to be multiplicative is that $\La = + \infty$ or $b_k \in
\vcomp[m\La - k\La]$, for $k = 0, \dotsc, m - 1$.
\end{proposition}
\begin{proof}
\textbf{\textsf{Necessity}}.
Since $a\in\vcompb[\La]$, if $\vcompbf$ is multiplicative, then $a^m \in
\vcompb[m\La]$. Hence, $b_0 + b_1 a + \dotsc + b_{m-1}a^{m-1} \in
\vcompb[m\La]$, and the conclusion follows from the definition
of~$\vcompb[m\La]$.
\sn
\textbf{\textsf{Sufficiency}}.
This is trivial if $\La = + \infty$.
Therefore, we can assume that $\La < +\infty$.

Take $q(X), q'(X) \in \KF[X]$ of degrees $n, n' < m$ respectively, and
cuts $\Lambda, \Lambda'$ with $q(a) \in \vcompbL$ and $q'(a) \in
\vcompb[\Lambda']$. We wish to show that $q(a)q'(a)\notin\vcompb[\Lambda
+ \Lambda']$. We write
\[\begin{aligned}
q(X)  & := \sum_{i=0}^n c_i X^i, \\
q'(X) & := \sum_{j=0}^{n'} c'_j X^j,
\end{aligned}\]
with $c_i \in \vcomp[\Lambda - i\La]$, $c'_j \in \vcomp[\Lambda' - j\La]$.
Define $e := n + n'-m \in \Zed$. We can assume that $e$ is minimal. If
$e < 0$, we have a contradiction with Proposition~\ref{PROP:RING-MULT}.
Hence, $e\geq 0$. Then,
\[
q(a)q'(a) = \sum_{i,j} c_i c'_j a^{i+j} =
\underbrace{\sum_{i+j < m+e} c_i c'_j a^{i+j}}_{S_1} +
\underbrace{a^{m+e} \sum_{i+j = m+e} c_i c'_j}_{S_2}.
\]
It suffices to prove that each summand in the sum above is in
$\vcompb[\Lambda + \Lambda']$. By definition of $\vcompbf$, $c_i a^i \in
\vcompbL$ and $c'_j a^j\in \vcompb[\Lambda']$. Therefore, by minimality
of $e$ and Proposition~\ref{PROP:N-GAMMA}, if $i + j < m + e$, then $c_i
c'_j a^{i+j} \in \vcompb[\Lambda + \Lambda']$, hence $S_1 \in
\vcompb[\Lambda + \Lambda']$.
\sn
Moreover, $a^{m+e} = -\sum_{k<m} b_k a^{k+e}$, hence
\begin{equation}\label{EQ:MULT-ALG}
S_2 = -\sum_{k<m} b_k a^{k+e} \sum_{i+j = m+e} c_i c'_j.
\end{equation}
It is enough to prove that $c_i c'_j b_k a^{k+e} \in \vcompb[\Lambda +
\Lambda']$ for each $i+j = m+e$, $k<m$. Fix $l, l' < m$ such that $l+l'
= k+e$. By Lemma~\ref{LEM:BASIC-MONOMIAL},
\[
a^l\in\vcompb[l \La]\>.
\]
By Lemma~\ref{LEM:D-K-M-I-J},
%
%
\begin{equation}\label{EQ:MULT-ALG-2}
c_i c'_j b_k \in \vcomp[(\Lambda - i \La) + (\Lambda' - j\La) +
(m\La - k\La)] \subseteq \vcomp[(\Lambda + \Lambda') - (k + e)\La].
\end{equation}
Hence, by Lemma~\ref{LEM:BASIC-MONOMIAL}
and Proposition~\ref{PROP:N-GAMMA},
\[
c_i c'_j b_k a^{l'}\in \vcompb[((\Lambda + \Lambda') - (k+d)\La)+ l'\La]
\subseteq \vcompb[\Lambda + \Lambda' - l\La].
\]
Since $l + l' < m + e$, we have that, by minimality of $e$ and
Proposition~\ref{PROP:N-GAMMA},
\[
c_i c'_j b_k a^{k+e} = a^l \cdot c_i c'_j b_k a^{l'} \in
\vcompb[((\Lambda + \Lambda') - l\La) + l\La] \subseteq
\vcompb[\Lambda + \Lambda'].
\qedhere
\]
%
\end{proof}
\sn
By Remark~\ref{REM:SUM-ITERATE}, in the case when $\La$ is of the form
$\gamma + \ig\La$ (in particular when it is the upper edge of a group),
the hypothesis of Proposition~\ref{PROP:ALG-MULT-NEC} is equivalent to
$b_k \in \vcomp[\La]$, for $k = 0, \dots, m - 1$.
\sn
An important example when $\La$ is a group is when $a$ is Henselian
over~$\KF$ (that is, $\val a \geq 0$, and $p(X)$, the minimal polynomial
of~$a$, has coefficients in~$\vring$, and $\val(\dot p(a)) = 0$),
cf.~\cite{DEL}.

\begin{lemma}\label{LEM:TOWER-HENSEL}
Let $\vcompf$ be a tower of complements for~$\KF$, and
$\KH$ be the Henselization of~$\KF$.
Then there is a tower of complements $\vcompbf$ for $\KH$
extending~$\vcompf$.
\end{lemma}
\begin{proof}
Let $\FF$ be a maximal subfield of $\KH$ such that:
\begin{itemize}
\item $\KF \subseteq \FF$;
\item there exists a tower of complements for $\FF$ extending $\vcompf$.
\end{itemize}
If $\KF = \KH$, we are done.
Otherwise, we will reach a contradiction.
\sn
\Wlog, we can assume that $\FF = \KF$. Since $\KH \neq \KF$, $\KF$
does not satisfy Hensel's Lemma and hence there exists $c \in \KH
\setminus \KF$ such that $c$ is Henselian over~$\KF$.

Let $p(X) \in \KF[X]$ be the minimal polynomial of $c$ over $\KF$, and
$\Lambda := \ball{c} = \sup \set{\val(c-d): d \in \KF}$. \Wlog, we can
assume that the degree of $p$ is minimal among the degrees of the
minimal polynomials of elements Henselian over $\KF$, and not in $\KF$.
Since $\val c \geq 0$, $\Lambda > 0$. Decompose $p(X) = p_1(X) +
p_2(X)$, with $p_1 \in \vcompL[X]$, $p_2 \in \vringL[X]$. Note that
$p_1$ is monic and of the same degree as $p$, while $\deg p_2 < \deg p$.
Moreover, $\val(\dot p_1(c)) = 0$, because $\val(\dot p(c)) = 0$, and
$\val(p_1 - p) > 0$.
\sn
Therefore, there exists $a\in\KH$ such that $p_1(a) = 0$,
$\val a \geq 0$ and $\val(c-a)>0$.
Thus, $a$ is Henselian over $\KF$.
\footnote{Since the minimal polynomial of $a$ is a divisor of $p_1$,
and, by Gau\ss's lemma, it is in $\vring[][X]$.}
Besides, by the minimality of the degree of~$p$, either $p_1$ is
irreducible, or $a\in\KF$.
\begin{claim}
$a \in \KH \setminus \KF$.
\end{claim}
In fact, by~\cite[Proposition~5.11]{F}\footnote{With $q := p_1$,
$\alpha = 0$.}, $\val(a-c) \geq \val(p - p_1) = \val(p_2) > \ballc$.
If $a$ were in~$\KF$, this would contradict the definition of~$\ballc$.
\sn
Therefore, $p_1$ is irreducible, and hence $p_1$ is the minimal
polynomial of $a$ over~$\KF$. Note moreover that $\ballc = \La$.
Finally, by the observation preceding this lemma, we can extend
$\vcompf$ to a tower of complements for $\KF(a)$, contradicting the
maximality of~$\KF$.
\end{proof}
\mn
Concluding, let $\FF$ be a Henselian valued field, with residue field
$\ff$ of characteristic~$0$, and value group~$G$. There exists a residue
field section $\iota: \ff \to \FF$, that we fix, and assume it is the
inclusion map. There also exists a value group section $t: G \to \FF^*$,
with factor set~$\factor$. We claim that there exists a \towcomp for
$\FF$ compatible with the choice of $\iota$ and of~$t$.
\begin{thm}\label{buildhens}
Let $\FF$ be a Henselian valued field, with residue field $\ff$ of
characteristic~$0$, and value group~$G$. Assume that $\FF$ contains its
residue field $\ff$ and let $t: G \to \FF^*$ be a section, with factor
set~$f$. Then there exists a \towcomp for $\FF$ compatible with the
inclusion of $\ff$ and with~$t$.
\end{thm}
\begin{proof}
We have seen that we can build a \towcomp for~$\ff(G, \factor)$. Let
$\KF \subseteq \FF$ be a maximal subfield admitting a \towcomp $\vcompf$
and such that $\ff(G, \factor) \subseteq \KF$. Then by
Lemma~\ref{LEM:TOWER-HENSEL}, $\KF$~is Henselian.
By Lemma~\ref{LEM:TOWER-TRASCENDENTAL}, $\FF$ is an algebraic extension
of~$\KF$. Since this extension is immediate and $\mathrm{char}\, \ff =
0$, this means that $\FF = \KF$.
\end{proof}

\subsection{The case of positive residue characteristic}   \label{PRC}
Towers of complements cannot always be extended to immediate algebraic
extensions.

\begin{example}\label{EX:WTOC}
Let $\ff$ be a perfect field of characteristic $p > 0$, and $y, t$ be
algebraically independent elements over~$\ff$. Let $\kf$ be the perfect
hull of $\ff(y)$, and $\KF$ be the perfect hull of $\kf(t)$, with
the $t$-adic valuation. $\KF$ has residue field $\kf$ and it is, in a
canonical way, a truncation-closed subfield of $\HF :=
\kf((t^{\frac{\Zed}{p^\infty}}))$.

Let $c$, $a$, $b$ be roots of the polynomials
\[\begin{aligned}
X^p - X &- y,\\
X^p - X &- \frac{1}{t},\\
X^p - X &- (\frac{1}{t} + y)
\end{aligned}
\]
respectively.
As an element of $\HF$, $a = t^{-\frac{1}{p}} + t^{-\frac{1}{p^2}}
+ t^{-\frac{1}{p^3}} + \dotsb$.
Moreover, $\kf(c)$ is a proper algebraic extension of $\kf$, and $b = a + c$.

It is easy to see that:
\begin{enumerate}
\item $\KF(a)$ and $\KF(b)$ are immediate algebraic extensions of $\KF$.
\item $\KF(a, b)$ is \emph{not} an immediate extension of $\KF$
(because $c\in\KF(a, b)$, but it is not in~$\kf$).
\end{enumerate}
Hence, $\KF$ has (at least) 2 different maximal immediate algebraic
extensions, one containing $a$, the other $b$. The truncation-closed
embedding of $\KF$ in $\HF$ can be extended to a truncation-closed
embedding of $\KF(a)$, but not of $\KF(b)$ (nor of any immediate
extension of $\KF(b)$).

However, both $a$ and $b$ satisfy the fundamental assumption b).
The truncation-closed embedding of $\KF$ in $\HF$ induces a unique tower
of complements~$\vcompf$. By Proposition~\ref{PROP:RING-EXT}, $\vcompf$
extends to a \wtoc for ~$\Kf(b)$, but no \wtoc for $\Kf(b)$ extending
$\vcompf$ can be a \towcomp Therefore, it is not possible to extend
$\vcompf$ to a tower of complements for $\KF(b)$, but only to a
\emph{weak} tower of complements.

The element $b$ does not satisfy the necessary and sufficient condition
of Proposition~\ref{PROP:ALG-MULT-NEC}, because $\La = \ball{b} = \sup_i
\set{-p^{-i}} =0^- < 0$, thus $\frac{1}{t} + y \notin \vcomp[p
\Lambda_b] = \vcomp[0^-]$.
\end{example}

This example together with the condition given in
Proposition~\ref{PROP:ALG-MULT-NEC}
leads us the way to construct extensions of towers of complements to at
least one suitable maximal immediate extension. For this, we need to
determine ``good'' minimal polynomials associated with pseudo Cauchy
sequences of algebraic type. From Theorem 13 of \cite{Ku} we infer the
following result, which is due to F.~Pop:

\begin{lemma}                               \label{MINPOLPC}
Assume that $\LF$ is a minimal immediate algebraic extension of the
henselian field $\KF$ of characteristic $p>0$, that is, it admits no
proper subextension. Then $\LF$ is generated by a root of an irreducible
polynomial of the form
\[p(X)\;=\;c+\sum_{i=0}^{n} c_i X^{p^i}\quad\mbox{with }\> c_n=1\;.\]
\end{lemma}
\n
Note that the polynomial
\[{\cal A}_p(X)\;:=\;p(X)-c\;=\;\sum_{i=0}^{n} c_i X^{p^i}\]
is additive, that is,
\[{\cal A}_p(x+y)\;=\;{\cal A}_p(x)+{\cal A}_p(y)\;.\]

We will now derive another minimal polynomial for an immediate extension
of $\KF$ from $p(X)$, one that is suitable for our purposes. Let $a$
denote the root of $p(X)$ that generates $\LF$, and choose a pseudo
Cauchy sequence $(a_\nu)_{\nu\in I}$ without a limit in $\KF$
that has $a$ as a limit. We have
\[-p(a_\nu)\>=\>p(a)-p(a_\nu)
\>=\> \sum_{i=1}^{n} c_i (a-a_\nu)^{p^i}\;.\]
Since the values $v(a-a_\nu)$ are eventually strictly increasing with
$\nu$, there is some $i_0$ (indepent of $\nu$) such that for all $i\ne
i_0\,$,
\[vc_{i_0} (a-a_\nu)^{p^{i_0}}\>=\>vc_{i_0} + p^{i_0} v(a-a_\nu)\><\>
\>vc_i + p^i v(a-a_\nu)\>=\>vc_i (a-a_\nu)^{p^i}\;.\]
We conclude that for large enough $\nu$, the values
\[vp(a_\nu)\>=\>v\sum_{i=0}^{n} c_i (a-a_\nu)^{p^i}
\>=\>vc_{i_0} (a-a_\nu)^{p^{i_0}}\]
are strictly increasing and are all contained in
\[vc_{i_0}+p^{i_0}\La\>\leq\>p^n\La\;,\]
where the inequality follows from the case $i=n$ in the above inequality
for the value of the summands.
For $0\leq i<n$, split
\[c_i\>=\>b_i+b'_i\]
at $p^n\La - p^i\La$ so that $b_i\in\vcomp[p^n\La - p^i\La]$ and
$b'_i\in\vring[p^n\La - p^i\La]$. As in the proof of
Proposition~\ref{PROP:RING-EXT} we find some
$\nu_0\in I$ such that for all $i$ and all $\nu>\nu_0\,$,
\[vb'_i+p^iv(a_\nu-a_{\nu_0})\> = \>vb'_i+p^iv(a-a_{\nu_0})\>
>\>p^n\La\;.\]
Now we split
\[c+{\cal A}_p(a_{\nu_0})\>=\>b+b'\]
at $p^n\La$ so that $b\in\vcomp[p^n\La]$ and $b'\in\vring[p^n\La]$.
We set
\[q(X)\>:=\> b+\sum_{i = 0}^n b_i X^{p^i}\in \KF[X]\;.\]
Then for all $\nu>\nu_0\,$,
\begin{eqnarray*}
q(a_\nu-a_{\nu_0}) & = & b+\sum_{i = 0}^n b_i (a_\nu-a_{\nu_0})^{p^i}\\
& = & c+\sum_{i=0}^n c_i (a_\nu-a_{\nu_0})^{p^i}+{\cal A}_p(a_{\nu_0})
-b'-\sum_{i=0}^n b'_i (a_\nu-a_{\nu_0})^{p^i}\\
& = & p(a_\nu)-b'-\sum_{i=0}^n b'_i (a_\nu-a_{\nu_0})^{p^i}\;.
\end{eqnarray*}
Since $vp(a_\nu) < p^n\La$ and $v(b'+\sum_{i=0}^n b'_i
(a_\nu-a_{\nu_0})^{p^i})>p^n\La$, we conclude that
\[vq(a_\nu-a_{\nu_0})\>=\>vp(a_\nu)\]
for all $\nu>\nu_0$. Hence, $vq(a_\nu-a_{\nu_0})$ is strictly increasing
for large enough $\nu$. On the other hand, $(a_\nu-a_{\nu_0})_{\nu\in I}$
is a pseudo Cauchy sequence without a limit in $\KF$, like
$(a_\nu)_{\nu\in I}$. Now we are ready to prove our main theorem for the
case of positive characteristic:
\begin{thm}                             \label{extamKf}
Assume that the valued field $\KF$ admits a tower of complements. Then
there is at least one maximal immediate extension and at least one
maximal immediate algebraic extension of $\KF$ that admits a tower of
complements that extends the tower of $\KF$.
\end{thm}
\begin{proof}
Take any immediate extension $\LF$ of $\KF$ that admits an extension of
the tower $\vcompf$ of $\KF$.
\sn
Claim: {\it If\/ $\LF$ is not algebraically maximal, then there is some
proper immediate algebraic extension of\/ $\LF$ that admits an extension
of the tower $\vcompf$ and thus of the tower of\/ $\KF$.}
\sn
We may assume that $\LF$ is henselian because otherwise, we are done by
Lemma~\ref{LEM:TOWER-HENSEL}.
Take $n$ to be minimal among all integers $>0$ for which $\LF$ admits an
immediate extension of degree $p^n$. Take $\LF'$ to be such an
extension of degree $p^n$, and take $p(X)$ as in Lemma~\ref{MINPOLPC}. Then
choose the pseudo Cauchy sequence $(a_\nu)_{\nu\in I}$ and construct the
polynomial $q(X)$ as described above. There cannot be any polynomial
$r(X)$ of degree $<p^n$ such that $vr(a_\nu-a_{\nu_0})$ eventually
increases with $\nu$ since otherwise by Theorem 3 of \cite{Kap}, there
would be an immediate extension of degree $<p^n$ of $\LF$ (note that any
immediate algebraic extension of $\LF$ has degree a power of $p$). Hence
again by
Theorem 3 of \cite{Kap}, we may choose any root $\tilde{a}$ of $q(X)$
and an immediate extension of
$v$ from $\LF$ to $\LF(\tilde{a})$. Since the coefficients of $q(X)$
satisfy the conditions of Proposition~\ref{PROP:ALG-MULT-NEC}, the tower
$\vcompf$ extends to a tower of complements of $\LF(\tilde{a})$.
This proves our claim and, by means of Zorn's Lemma, the algebraic part
of our theorem.
\par\smallskip
If $\LF$ is not maximal, then by what we have shown, it admits an
extension of the tower to some maximal immediate algebraic extension. If
it is already algebraically maximal but admits a proper immediate
extension generated by a limit $a$ of a transcendental pseudo Cauchy
sequence, then Lemma~\ref{LEM:TOWER-TRASCENDENTAL} shows that the tower
of $\LF$ can be extended to a tower of $\LF(a)$. Again by means of
Zorn's Lemma, this proves the remaining part of our theorem.
\end{proof}
Since the maximal immediate algebraic extensions of Kaplansky fields are
unique up to (valuation preserving) isomorphism, we obtain:
\begin{thm}\label{buildkapl}
Let $\FF$ be an algebraically maximal Kaplansky field. Assume that $\FF$
contains its residue field $\ff$ and let $t: G \to \FF^*$ be a section,
with factor set~$f$. Then there exists a \towcomp for $\FF$ compatible
with the inclusion of $\ff$ and with~$t$.
\end{thm}

Since algebraically maximal Kaplansky fields of positive characteristic,
being perfect fields, always admit embeddings of their residue field and
a value group section, we conclude:
\begin{corollary}                             \label{tcepc}
Every algebraically maximal Kaplansky field of positive characteristic
with value group $G$ and residue field $\kf$ admits a truncation closed
embedding in some power series field $\kf((G, f))$. In particular, every
algebraically closed valued field of positive characteristic with value
group $G$ and residue field $\kf$ admits a truncation closed
embedding in $\kf((G))$.
\end{corollary}



\adresse
\end{document}